\documentclass[11pt,reqno]{amsart}
\usepackage{tikz}
\usepackage{bm}
\usepackage{subfig}
\usepackage{hyperref}
\usepackage{epsfig}
\usepackage{math}
\usepackage[norelsize,boxed,linesnumbered,vlined,algo2e]{algorithm2e} 
\graphicspath{{figures/}}
\usepackage{tikz}
\usetikzlibrary{shapes,arrows}

\usepackage{algorithm}
\usepackage{algorithmic}

\usepackage{media9}
\usepackage{graphicx}

\tikzstyle{decision} = [diamond, draw, fill=blue!20, 
    text width=4.5em, text badly centered, node distance=3cm, inner sep=0pt]
\tikzstyle{block} = [rectangle, draw, fill=blue!20, 
    text width=5em, text centered, rounded corners, minimum height=4em]
\tikzstyle{line} = [draw, -latex']
\tikzstyle{cloud} = [draw, ellipse,fill=red!20, node distance=3cm,
    minimum height=2em]
\usetikzlibrary{positioning}
\tikzset{main node/.style={circle,fill=blue!20,draw,minimum size=1cm,inner sep=0pt},  }

\def\v{\bm{v}}
\def\m{\bm{m}}

\def\d{\mathrm{d}}

\usepackage[revision]{revdiff}

\begin{document}
\title[]{Generalized Unnormalized Optimal Transport and its fast algorithms}
\author[Lee]{Wonjun Lee}
\author[Lai]{Rongjie Lai}
\author[Li]{Wuchen Li}
\author[Osher]{Stanley Osher}
\email{wlee@math.ucla.edu}
\email{lair@rpi.edu}
\email{wcli@math.ucla.edu}
\email{sjo@math.ucla.edu}
\address{Department of Mathematics, Rensselaer Polytechnic Institute.}
\address{Department of Mathematics, University of California, Los Angeles.}
\thanks{W. Lee, W. Li and S. Osher's research are supported in part by AFOSR MURI FA9550-18-1-0502. R. Lai's reserach is supported in part by an NSF CAREER Award DMS--1752934.}
\maketitle

\begin{abstract}
We introduce fast algorithms for generalized unnormalized optimal transport. To handle densities with different total mass, we consider a dynamic model, which mixes the $L^p$ optimal transport with $L^p$ distance. For $p=1$, we derive the corresponding $L^1$ generalized unnormalized Kantorovich formula. We further show that the problem becomes a simple $L^1$ minimization which is solved efficiently by a primal-dual algorithm. For
$p=2$, we derive the $L^2$ generalized unnormalized Kantorovich formula, a new unnormalized Monge problem and the corresponding Monge-Amp\`ere equation. Furthermore, we introduce a new unconstrained optimization formulation of the problem. The associated gradient flow is essentially related to an elliptic equation which can be solved efficiently. Here the proposed gradient descent procedure together with the Nesterov acceleration involves the Hamilton-Jacobi equation which arises from the KKT conditions. Several numerical examples are presented to illustrate the effectiveness of the proposed algorithms.
\end{abstract}

\section{Introduction}
\label{sec:intro}
Optimal transport describes transport plans and metrics between two densities with equal total mass \cite{Villani2009_optimal}. It has wide applications in various fields such as physics \cite{LegerLi2019_hopfcole,LiSBP}, mean field games \cite{HJD}, image processing \cite{PeyreCuturi2018_computationala}, economics \cite{beck2009fast}, inverse problem \cite{Yang2,Yang1}, Kalman filter \cite{Garbuno-InigoHoffmannLiStuart2019_interacting} as well as machine learning \cite{WGAN,Wproximal}. In practice, it is also natural to consider transport and metrics between two densities with different total mass. For example, in image processing, it is very common that we need to compare and process images with unequal total intensities \cite{Li2}. 

Recently, there has been increasing interests in studying the optimal transport between two densities with different total mass. Based on the linear programming formulation, generalized versions for unnormalized optimal transport have been considered in \cite{PR2,Thorpe:2017:TLD:3140477.3140524}. In this paper, our discussion is based on the fluid-dynamic formulation following \cite{BenamouBrenier2000_computational}, which has significantly fewer variables than the linear programming formulation. In this work, a source function is considered to provide dynamical behaviors of a source term during transportation. Adding a source term for handling densities with unequal total mass has been considered in \cite{CL,chen2019interpolation,WF,WF2,Mielke,Maas,PR}. 
These methods consider density-dependent source terms and lead to a dynamical mixture of Wasserstein-2 distance and Fisher-Rao distance. The corresponding minimization of the source term is weighted with the density. 
More recently, a spatially independent source function was considered in \cite{gangbo2019unnormalized} to transport densities with unequal mass. This model results in creating or removing masses in the space uniformly during transportation when moving one density to another. Here, we further extend the model \cite{gangbo2019unnormalized} using a spatially dependent source function. As a result, the transportation map between two densities with different masses has the flexibility to create or remove masses locally. 
In all our models, the source term does not depend on the current density. This property keeps the Hamilton-Jacobi equation arising in the original (normalized) optimal transport problem. We further explore the Kantorovich duality and derive the corresponding unnormalized Monge problems and Monge-Amp\`ere equations.
Besides these model derivations, the other main contribution of this paper is to propose fast algorithms for all related dynamical optimal transport problems with source terms. 

More specifically, the proposed model is a minimal flux problem mixing both $L^p$ metric and Wasserstein-$p$ metric, following Benamou-Brenier formula \cite{BenamouBrenier2000_computational}. In particular, we focus on the cases $p=1$ and $p=2$, and design corresponding fast algorithms. For the $L^1$ case, we propose a primal-dual algorithm \cite{CP}. The method updates variables at each iteration with explicit formulas, which only involve low computational cost shrink operators, such as those used in \cite{LiRyuOsherYinGangbo2018_parallel}. For the $L^2$ case, we formulate the minimal flux problem into an unconstrained minimization problem as follows
\begin{equation}\label{a}
\begin{split}
\inf_{\mu} \Biggl\{ &\int^1_0 \int_\Omega 
         \partial_t \mu(t,x) (-\nabla\cdot(\mu(t,x)\nabla) + \alpha \textrm{Id})^{-1} \partial_t \mu(t,x) \d x \d t:\\
& \quad\quad           \mu(0,x) = \mu_0(x), \mu(1,x) = \mu_1(x), x\in\Omega
            \Biggl\},
\end{split}
\end{equation}
where $\alpha$ is a given positive scalar, Id is the identity operator, and the infimum is taken among all density paths $\mu(t,x)$ with fixed terminal densities $\mu_0$, $\mu_1$. From the associated Euler-Lagrange equation, we derive a Nesterov accelerated gradient descent method to solve the unnormalized optimal transport problem. It turns out that our method only needs to solve an elliptic equation involving the density at each iteration. Thus, fast solvers for elliptic equations can be directly used. Interestingly, the Euler-Lagrange equation of this formulation introduces the Hamilton-Jacobi equation, which characterizes the Lagrange multiplier (see related studies in \cite{LiG}). 
We, in fact, construct the gradient descent method in the density path space to solve this equation:
\begin{equation*}
\partial_\tau\mu(\tau,t,x)=\partial_t\Phi(\tau,t,x)+\frac{1}{2}\|\nabla\Phi(\tau,t,x)\|^2,
\end{equation*}
with
\begin{equation*}
\Phi(\tau,t,x)=(-\nabla\cdot(\mu(\tau,t,x)\nabla)+\alpha Id)^{-1}\partial_t\mu(\tau,t,x).
\end{equation*}
Here $\tau$ is an artificial time variable in optimization. The minimizer path $\mu^*(t,x)$ is obtained by  solving $\mu^*(t,x)=\lim_{\tau\rightarrow\infty}\mu(\tau,t,x)$ numerically.

The outline of this paper is as follows. In section \ref{section:prob-statement}, we propose a formulation for the generalized unnormalized optimal transport. We then derive the Kantorovich duality for both cases. We also formulate the generalized unnormalized Monge problem and the corresponding Monge-Amp\`ere equation. In section \ref{section:numerical-methods}, we propose a fast algorithm for $L^1$-generalized unnormalized optimal transport using a primal-dual based method. We also propose a new method for $L^2$-generalized unnormalized optimal transport based on the Nesterov accelerated gradient descent method. In addition, we discuss detailed numerical discretization of the two problems. In section \ref{section:numerical-experiments}, we present several numerical experiments to demonstrate our algorithms. We conclude the paper in section \ref{sec:con}.


\section{Generalized unnormalized optimal transport}\label{section:prob-statement}
In this section, we study a formulation of generalized unnormalized optimal transport problem as a natural extension of the exploration studied in \cite{gangbo2019unnormalized}. We specifically discuss the $L^1$ and $L^2$ versions of the generalized unnormalized optimal transport and their associated Kantorovich dualities. Furthermore, we derive a new generalized unnormalized Monge problem and the corresponding Monge-Amp\`ere equation.

Let $\Omega \subset \RR^d$ be a compact convex domain. Denote the space of unnormalized densities $\cm(\Omega)$ by
$$\cm(\Omega) := \{ \mu \in L^1(\Omega): \mu(x) \geq 0 \}.$$

Given two densities $\mu_0,\mu_1 \in \cm(\Omega)$, we define the generalized unnormalized optimal transport as follows:
\begin{definition}[Generalized Unnormalized Optimal Transport] \label{def:p}
    Define the $L^p$ generalized unnormalized Wasserstein distance $UW_p: \cm(\Omega) \times \cm(\Omega) \rightarrow \RR$ by
    \begin{align*}
        UW_p(\mu_0,\mu_1)^p
        = \inf_{\v,\mu,f} \int^1_0 \int_\Omega \|\v(t,x)\|^p \mu (t,x) \d x\d t + \frac{1}{\alpha} \int^1_0\int_\Omega |f(t,x)|^p \d x\d t,
    \end{align*}
    such that the dynamical constraint, i.e. the unnormalized continuity equation, holds
    \begin{align*}
        \partial_t \mu(t,x) + \nabla \cdot (\mu(t,x)\v(t,x)) = f(t,x),\quad\mu(0,x) = \mu_0(x),\quad \mu(1,x) = \mu_1(x).
    \end{align*}
    The infimum is taken over continuous unnormalized density functions $\mu: [0,1] \times \Omega \rightarrow \RR$, and Borel vector fields $v:[0,1] \times \Omega \rightarrow \RR^d$ with zero flux condition on $[0,1]\times \partial \Omega$, and Borel spatially dependent source functions $f:[0,1] \times \Omega \rightarrow \RR$ with a positive constant $\alpha$.
\end{definition}
This is a generalized definition of unnormalized optimal transport from \cite{gangbo2019unnormalized}. Here, we consider a spatially dependent source function $f(t,x)$. In this paper, we will focus on the cases with $p=1$ and $p=2$. 

\begin{remark}
We note that \cite{chen2019interpolation} has proposed the model for $p=2$ without any discussion about numerical methods. In this paper, we mainly study Kantorovich duality and design fast algorithms. 
\end{remark}
\begin{remark}
In literature, \cite{WF} studied the other dynamical formulations of unbalanced optimal transport problems. In their approach, the optimal source term is expressed as a product of a density function and a scalar field function.
In our approach, the optimal source term only depends on a scalar field function. 
This fact shows that our approach is different from \cite{WF} in variational problems and dual (Kantorovich) problems.
\end{remark}

\subsection{$L^1$ Generalized Unnormalized Wasserstein metric.} When $p=1$, the problem (\ref{def:p}) becomes
\begin{equation}\label{def-L1}
\begin{split}
    UW_1(\mu_0, \mu_1) = \inf_{\v,\mu,f}\Biggl\{ \int^1_0 \int_\Omega \|\v(t,x)\| \mu (t,x) \d x\d t + \frac{1}{\alpha} \int^1_0\int_\Omega |f(t,x)| \d x\d t:&\\
    \partial_t \mu(t,x) + \nabla \cdot (\mu(t,x)\v(t,x)) = f(t,x)& \\
    \quad\mu(0,x) = \mu_0(x),\quad\mu(1,x) = \mu_1(x)&
    \Biggl\}.
\end{split}
\end{equation}
Here $\|\cdot\|$ can be any homogeneous of degree one norm, i.e. $l_q$ norm. E.g., $\|u\|_q=(\sum_{i=1}^d|u_i|^q)^{\frac{1}{q}}$. In particular, we consider $q=1,2$ with 
\begin{align*}
    \|u\|_1 = |u_1| + \cdots + |u_d|\indent \text{ for $ u \in \RR^d$},
\end{align*} 
or 
\begin{align*}
    \|u\|_2 = \sqrt{|u_1|^2 + \cdots + |u_d|^2}\indent \text{ for $ u \in \RR^d$}.
\end{align*} 

\begin{proposition}\label{prop:l1}
    The $L^1$ unnormalized Wasserstein metric is given by
    \begin{align}\label{eq:prop-l1-m-c}
        UW_1(\mu_0,\mu_1) = \inf_{\m,c} \Biggl\{ 
        & \int_\Omega \|\m(x)\| \d x + \frac{1}{\alpha} \int_{\Omega} |c(x)| \d x ~:~ \nonumber\\
        &\mu_1(x) - \mu_0(x) + \nabla \cdot \m(x) - c(x) = 0
    \Biggl\}.
    \end{align}
    There exists $\Phi(x)$, such that the minimizer $(m,c)$ for the problem (\ref{eq:prop-l1-m-c}) satisfies
$$ \nabla\Phi(x) \in \partial \|\m(x)\| \quad \text{and} \quad \alpha \Phi(x) \in \partial |c(x)|$$
where $\partial \|\m(x)\|$ and $\partial |c(x)|$ denote their sub-differentials. 
\end{proposition}
\begin{proof}

\noindent Denote
\begin{align*}
    \m(x) = \int^1_0 \v(t,x) \mu(t,x) \d t,
\end{align*}
Using Jensen's inequality and integration by parts, we can reformulate (\ref{def-L1}).
\begin{equation}
    \begin{split}
            &\int^1_0 \int_\Omega \|\v(t,x)\| \mu(t,x) \d x \d t + \frac{1}{\alpha} \int^1_0 \int_{\Omega} |f(t,x)| \d x \d t
        \\
            &\geq \int_\Omega \|\m(x)\| \d x + \frac{1}{\alpha} \int_{\Omega} \left|\int^1_0  f(t,x) \d t \right| \d x.
         \label{eq:jensen-2}
    \end{split}
\end{equation}


\noindent Define $c(x) = \int^1_0 f(t,x) \d t$. Integrating on the constraint of problem \eqref{def-L1} with the zero flux condition of $v$ yields,
\begin{align*}
    \int_\Omega c(x) \d x = \int^1_0 \int_\Omega f(t,x) \d x \d t = \int_\Omega \mu_1(x)\d x - \int_\Omega \mu_0(x) \d x.
\end{align*}
Plug $c(x)$ into the equation (\ref{eq:jensen-2}), we obtain a new formulation.
\begin{align*}
    \inf_{\m,c} \Biggl\{  
        \int_\Omega \|\m(x)\| \d x + \frac{1}{\alpha} \int_{\Omega} \|c(x)\| \d x ~:~ 
        &\mu_1(x) - \mu_0(x) + \nabla \cdot \m(x) - c(x) = 0
    \Biggl\}.
\end{align*}

Note that the minimization path can be attained in the inequality \eqref{eq:jensen-2} by choosing $\mu(t,x) = t \mu_0(x) + (1-t) \mu_1(x)$, $\m(x) = \mu(t,x) \v(t,x)$ and $f(t,x) = c(x)$. Then $\{\mu(t,x), \v(t,x),f(t,x)\}$ is a feasible solution to (\ref{def-L1}) and (\ref{eq:prop-l1-m-c}) , hence the two minimization problems have the same optimal value.

Consider the Lagrangian of this minimization problem.
\begin{align}\label{prop:l1-lagrangian}
    \cl (\m,c,\Phi) &= \int_\Omega \|\m(x)\|\d x + \frac{1}{\alpha}\int_\Omega |c(x)|\d x + \int_\Omega \Phi(x)\biggl( \mu_1(x) - \mu_0(x) + \nabla\cdot \m(x) - c(x) \biggl),
\end{align}
where $\Phi(x)$ is a Lagrange multiplier. From the Karush\textendash Kuhn\textendash Tucker (KKT) conditions, we derive the following properties of the minimizer
\begin{align*}
    0\in \partial_{\m} \cl &\Rightarrow  \nabla \Phi(x)\in \partial \|\m(x)\| \\
    0\in \partial_c \cl  &\Rightarrow  \alpha \Phi(x) \in \partial |c(x)| \\
    \delta_{\Phi} \cl =0 &\Rightarrow \mu_1(x) - \mu_0(x) + \nabla\cdot \m(x) - c(x) = 0.
\end{align*}

\end{proof}

\begin{remark}
    In the case that $L^1$ unnormalized Wasserstein metric with a spatially independent function $f(t)$, $c$ is defined to be $c = \int^1_0 f(t) \d t$, which is a constant. Integrating on a spatial domain for continuity equation,
    \begin{equation*}
        c = \frac{1}{|\Omega|}\biggl(\int_\Omega \mu_0(x) \d x - \int_\Omega \mu_0 (x) \d x \biggl).
    \end{equation*}
    As a result, the minimization problem becomes
    \begin{equation*}
        \begin{split}
        UW_1(\mu_0,\mu_1) = \inf_{\m} \biggl\{&
            \int_\Omega \|\m(x)\|\d x + \frac{1}{\alpha} \biggl| \int_\Omega \mu_1(x)\d x - \int_\Omega \mu_0(x) \d x \biggl|:\\
            &\mu_1(x) - \mu_0(x) + \nabla \cdot \m(x) = \frac{1}{|\Omega|}\biggl(\int_\Omega \mu_1(x)\d x - \int_\Omega \mu_0(x) \d x\biggl)
        \biggl\}.
        \end{split}
    \end{equation*}
    This is compatible with the result obtained in \cite{gangbo2019unnormalized}. In this case, we note that $\m(x)$ does not depend on $\alpha$.
\end{remark}

\begin{proposition}[$L^1$ Generalized Unnormalized Kantorovich formulation]\label{prop:l1-kantorovich}
    The Kantorovich formulation of $L^1$ unnormalized Wasserstein metric is the following:
    \begin{align}
        UW_1(\mu_0,\mu_1) = \sup_\Phi \biggl\{
            \int_\Omega\Phi(x)(\mu_1(x) - \mu_0(x))\d x : \|\nabla \Phi\| \leq 1, |\Phi| \leq \frac{1}{\alpha}
        \biggl\}
    \end{align}
    
\end{proposition}
\begin{remark}
The Kantorovich formulation of the generalized unnormalized Wasserstein-1 metric has also been stated in \cite{chen2017matricialW1} for the $\|\cdot\|_2$ norm. \end{remark}
\begin{proof}
    From the Lagrangian (\ref{prop:l1-lagrangian}),
    \begin{align*}
        &\inf_{\m,c} \sup_\Phi \cl(\m,c,\Phi)\\
        &\geq \sup_\Phi \inf_{\m,c} \cl(\m,c,\Phi)\\
        &= \sup_\Phi \inf_{\m,c} \biggl\{ \int_\Omega \|\m(x)\|\d x + \frac{1}{\alpha} \int_\Omega |c(x)| \d x + \int_\Omega \Phi(x) (\mu_1(x) - \mu_0(x) + \nabla \cdot \m(x) - c(x)) \d x\biggl\}\\
        &= \sup_\Phi \inf_{\m,c} \biggl\{ \int_\Omega \|\m(x)\|\d x + \frac{1}{\alpha} \int_\Omega |c(x)| \d x + \int_\Omega \Phi(x) (\mu_1(x) - \mu_0(x) - c(x)) \d x \\
        & \qquad - \int_\Omega \nabla \Phi(x) \cdot \m(x) \d x + \int_{\partial \Omega} \Phi(x) \m(x) \cdot n(x) ds(x) \biggl\}\\
        &= \sup_\Phi \biggl\{\int_\Omega \Phi(x) (\mu_1(x) - \mu_0(x)) + \inf_{\m,c} \int_\Omega \|\m(x)\| - \nabla \Phi(x) \cdot \m(x) \d x + \int_\Omega \frac{1}{\alpha} |c(x)| - \Phi(x) c(x) \d x \biggl\}\\
        &= \sup_\Phi \biggl\{
            \int_\Omega\Phi(x)(\mu_1(x) - \mu_0(x))\d x : \|\nabla \Phi\| \leq 1, |\Phi| \leq \frac{1}{\alpha}
        \biggl\}.
    \end{align*}
    From the calculation, the optimizer $\Phi$ satisfies the following:
    \begin{align*}
         \nabla \Phi \in\partial \|\m(x)\|, \indent \alpha \Phi \in \partial |c(x)|.
    \end{align*}
    We show the duality gap is zero using the proposition \ref{prop:l1}.
    \begin{align*}
        &\int_\Omega \|\m(x)\|\d x + \frac{1}{\alpha} \int_\Omega |c(x)| \d x + \int_\Omega \Phi(x) (\mu_1(x) - \mu_0(x) + \nabla \cdot \m(x) - c(x)) \d x\\
        &=\int_\Omega \|\m(x)\| - \nabla \Phi \cdot \m(x) \d x +  \int_\Omega \frac{1}{\alpha} |c(x)| - \Phi(x) c(x) \d x + \int_\Omega \Phi(x) (\mu_1(x) - \mu_0(x)) \d x\\
        &=\int_\Omega \Phi(x) (\mu_1(x) - \mu_0(x)) \d x
    \end{align*}
    This concludes the proof.
\end{proof}

\subsection{$L^2$ Generalized Unnormalized Wasserstein metric.}
Let $p=2$. From the definition (\ref{def:p}), we now consider
\begin{equation}\label{eq:uw2}
\begin{aligned}
    UW_2(\mu_0,\mu_1)^2
    = \inf_{v,\mu,f} \biggl\{ &\int^1_0 \int_\Omega \|\v(t,x)\|^2 \mu (t,x) \d x\d t + \frac{1}{\alpha} \int^1_0\int_\Omega \|f(t,x)\|^2 \d x\d t :\\
    &\partial_t \mu(t,x) + \nabla \cdot (\mu(t,x) \v(t,x)) = f(t,x),  t\in[0,1], x\in \Omega,\\
    &\mu(0,x) = \mu_0(x), \mu(1,x) = \mu_1(x)
    \biggl\}.
\end{aligned}
\end{equation}

\begin{proposition}\label{prop:l2}
    The $L^2$ generalized unnormalized Wasserstein metric is a well-defined metric function in $M(\Omega)$. In addition, the minimizer $(\v(t,x),\mu(t,x),f(t,x))$ for (\ref{eq:uw2}) satisfies
    \begin{align*}
        \v(t,x) = \nabla \Phi(t,x),\indent f(t,x) = \alpha \Phi(t,x),
    \end{align*}
    and
    \begin{align*}
        & \partial_t \mu(t,x) + \nabla \cdot (\mu(t,x) \nabla \Phi(t,x)) = \alpha \Phi(t,x)\\
        & \partial_t \Phi(t,x) + \frac{1}{2} \|\nabla \Phi(t,x)\|^2 \leq 0.
    \end{align*}
    In particular, if $\mu(t,x) >0$, then
    \begin{align*}
        \partial_t\Phi(t,x) + \frac{1}{2} \|\nabla \Phi(t,x)\|^2 = 0.
    \end{align*}
\end{proposition}

\begin{proof}
Denote $\m(t,x) = \mu(t,x)\v(t,x)$. Then the problem becomes
\begin{equation}\label{eq:UW2-m-mu}
\begin{split}
    \frac{1}{2}UW_2(\mu_0,\mu_1)^2 = \inf_{\m,\mu,f} \biggl\{ &\int^1_0 \int_\Omega \frac{\|\m(t,x)\|^2}{2\mu(t,x)} \d x\d t + \frac{1}{2\alpha} \int^1_0 \int_\Omega |f(t,x)|^2 \d x \d t:\\
   & \partial_t \mu(t,x) + \nabla\cdot \m(t,x) = f(t,x),\\ 
   &\mu(0,x) = \mu_0(x), \mu(1,x) = \mu_1(x),x\in \Omega, 0\leq t \leq 1 \biggl\}.
\end{split}
\end{equation}
Denote $\Phi(t,x)$ as a Lagrange multiplier. Consider the Lagrangian
\begin{align*}
    \cl (\m,\mu,f,\Phi) &= \int^1_0 \int_\Omega \frac{\|\m(t,x)\|^2}{2\mu(t,x)} \d x\d t + \frac{1}{2\alpha} \int^1_0 \int_\Omega |f(t,x)|^2 \d x \d t\\
    &+ \int^1_0 \int_\Omega \Phi(t,x) \Big( \partial_t \mu(t,x) + \nabla\cdot \m(t,x) - f(t,x) \Big) \d x \d t.
\end{align*}
From KKT condition $\delta_{\m} \cl=0, \delta_\mu \cl \geq 0, \delta_f \cl=0, \delta_\Phi \cl=0$, the minimizer satisfies the following properties:
\begin{align}
    &\frac{\m(t,x)}{\mu(t,x)} = \nabla \Phi(t,x)\label{prop:l2-eq1}\\
    &-\frac{\|\m(t,x)\|^2}{2\mu(t,x)^2} - \partial_t \Phi(t,x) \geq 0\label{prop:l2-eq2}\\
    &f(t,x) = \alpha \Phi(t,x)\nonumber\\
    &\partial_t \mu(t,x) + \nabla \cdot \m(t,x) - f(t,x) = 0\nonumber.
\end{align}
Combining (\ref{prop:l2-eq1}) and (\ref{prop:l2-eq2}) yields: 
$
    \partial_t \Phi(t,x) + \frac{1}{2} \|\nabla \Phi(t,x)\|^2 \leq 0.
$
\end{proof}

We next derive the corresponding Monge problem for unnormalized optimal transport with a spatially dependent source function. We note that the following derivations are formal in Eulerian coordinates of fluid dynamics. The related rigorous proof can be shown in Lagrangian coordinates similar as the one in \cite{Villani2009_optimal}.

\begin{proposition}[Generalized Unnormalized Monge problem]
    \begin{equation}
        \begin{split}
            &UW_2(\mu_0,\mu_1)^2 = \inf_{M, f(t,x)} \int_\Omega \|M(x) - x\|^2 \mu_0(x) \d x + \alpha \int^1_0 \int_\Omega |f(t,x)|^2 \d x \d t\\
            & + \int_\Omega \int^1_0 \int^t_0 f\biggl(s,sM(x) + (1-s)x\biggl) \|M(x) - x\|^2 Det\biggl(s\nabla M(x) + (1-s) \II\biggl) \d s \d t \d x
        \end{split}
    \end{equation}
    where $M:\Omega \rightarrow \Omega$ is an invertible mapping function and $f:\Omega \times [0,1] \rightarrow \RR$ is a spatially dependent source function. The unnormlized push forward relation holds
    \begin{equation}\label{eq:monge-push-forward}
        \begin{split}
        &\mu(1,M(x))Det(\nabla M(x))\\
        &= \mu(0,x) + \int^1_0 f\biggl(t,tM(x) + (1-t)\II\biggl) Det\biggl(t\nabla M(x) + (1-t)\II\biggl) \d t.
        \end{split}
    \end{equation}
\end{proposition}

\begin{proof}
We derive the Lagrange formulation of the unnormalized optimal transport with $p=2$. Consider a mapping function $X_t(x)$ with vector field $\v(t, X_t(x))$ satisfying
\begin{align}
    \frac{\d}{\d t} X_t(x) = \v(t, X_t(x)), \indent X_0(x) = x.
\end{align}
Then
\begin{align}
        \int_\Omega \int^1_0 \|\v(t,x)\|^2 \mu(t,x) \d t \d x &=
            \int_\Omega \int^1_0 \|\v(t, X_t(x))\|^2 \mu(t,X_t(x)) Det(\nabla X_t(x)) \d x \d t\nonumber\\
            &=
            \int_\Omega \int^1_0 \|    \frac{\d}{\d t}  X_t(x)\|^2 \mu(t,X_t(x)) Det(\nabla X_t(x)) \d x \d t.\label{eq:monge}
\end{align}
Define $J(t,x) := \mu(t,X_t(x))Det\bigl(\nabla X_t(x)\bigl)$. Differentiate $J(t,x)$ with respect to $t$,

\begin{align*}
    \frac{\d}{\d t} J(t,x) &= \frac{\d}{\d t} \biggl\{ \mu(t,X_t(x)) Det(\nabla X_t(x)) \biggl\}\\
    &= \partial_t \mu(t,X_t(x))Det(\nabla X_t(x)) + \nabla_X \mu(t,X_t(x))\cdot \frac{\d}{\d t} X_t(x) Det(\nabla X_t(x))\\
    &\hspace{1cm}+ \mu(t,X_t(x)) \partial_t Det(\nabla X_t(x))\\
    &= \partial_t \mu(t,X_t(x))Det(\nabla X_t(x)) + \nabla_X \mu(t,X_t(x))\cdot \frac{\d}{\d t} X_t(x) Det(\nabla X_t(x))\\
    &\hspace{1cm}+ \mu(t,X_t(x)) \nabla \cdot \v(t,X_t(x)) Det(\nabla X_t(x))\\
    &= \biggl( \partial_t\mu + \v \cdot \nabla \mu + \mu \nabla \cdot \v \biggl)(t,X_t(x)) Det(\nabla X_t(x))\\
    &= \biggl( \partial_t\mu + \nabla \cdot(\mu \v) \biggl)(t,X_t(x)) Det(\nabla X_t(x))\\
    &= f\bigl(t,X_t(x)\bigl) Det(\nabla X_t(x)).
\end{align*}
Denote
\begin{align*}
    J(t,x) = J(0,x) + \int^t_0    \frac{\d}{\d s} J(s,x)\d s.
\end{align*}
Since $X_0(x) = x$ and $\nabla X_0(x) = \II$, then $J(0,x) = \mu(0,x)$. This yields
\begin{align*}
    \mu(t,X_t(x)) Det(\nabla X_t(x)) = \mu(0,x) + \int^t_0 f\bigl(s,X_s(x)\bigl) Det (\nabla X_s(x)) \d s.
\end{align*}
Since the minimizer in Eulerian coordinates satisfies the Hamilton-Jacobi equation:
\begin{equation*}
    \partial_t\Phi(t,x) + \frac{1}{2}\|\nabla \Phi(t,x)\|^2 = 0,
\end{equation*}
and $\frac{\d}{\d t} X_t(x) = \nabla \Phi(t, X_t(x))$, then we have $    \frac{\d^2}{\d t^2} X_t(x)=0$. This implies
\begin{equation*}
    \frac{\d}{\d t}X_t(x) = \v(t,X_t(x)) = M(x) - x,
\end{equation*}
thus $X_t(x) = (1-t)x + t M(x)$ and $Det(\nabla X_t(x))=Det((1-t)\II + t\nabla M(x)).$
Substitute all the above into (\ref{eq:monge}):
\begin{align*}
    (\ref{eq:monge}) &= \int^1_0 \int_\Omega \|\frac{\d}{\d t}X_t(x)\|^2 J(t,x)\d x \d t\\
    &= \int^1_0 \int_\Omega \|M(x) - x \|^2 \biggl( J(0,x) + \int^t_0    \frac{\d}{\d s} J(s,x)ds\biggl) \d x \d t\\
    &= \int^1_0 \int_\Omega \|M(x) - x\|^2 \mu(0,x) \d x \d t \\
    &+ \int^1_0 \int_\Omega \|M(x) - x \|^2 \int^t_0 f\bigl(s,X_s(x)\bigl) Det(\nabla X_s(x))\d s \d x \d t\\
    &= \int_\Omega \|M(x) - x\|^2 \mu(0,x) dx \\
    &+ \int^1_0 \int^t_0 \int_\Omega \|M(x) - x\|^2 f\biggl(s,sM(x)+(1-s)x\biggl) Det\biggl((1-s)\II + s\nabla M(x)\biggl) \d x \d s \d t.
\end{align*}
This concludes the derivation.
\end{proof}

We next find the relation between the spatially dependent source function $f(t,x)$ and the mapping function $M(x)$. For the simplicity of presentation, here we assume the periodic boundary conditions on $\Omega$. 

\begin{proposition}[Generalized Unnormalized Monge-Amp\`ere equation]
    The optimal mapping function $M(x) = \nabla \Psi(x)$ satisfies the following unnormalized Monge-Amp\`ere equation
   \begin{equation}\label{NMA}
    \begin{split}
        &\mu(1,\nabla\Psi(x))Det(\nabla^2 \Psi(x)) - \mu(0,x) \\
        &= \alpha \int^1_0  \biggl(\Psi(x) - \frac{\|x\|^2}{2} + \frac{t\|\nabla \Psi(x) - x\|^2}{2}\biggl) Det\Big(t\nabla^2\Psi(x) + (1-t)\II\Big) \d t.
    \end{split}
    \end{equation}
\end{proposition}
\begin{proof}
    From the Hopf-Lax formula for the Hamilton-Jacobi equation,
    \begin{equation*}
        \Phi(1,M(x)) = \Phi(0,x) + \frac{\|M(x) - x\|^2}{2}.
    \end{equation*}
    Thus $\nabla \Phi(0,x) + x - M(x) = 0$. We further denote $\Psi(x) = \Phi(0,x) + \frac{\|x\|^2}{2}$, then $M(x) = \nabla \Psi(x)$. From $X_t(x) = (1-t)x + tM(x)$, then
    \begin{align*}
        \Phi(t,X_t(x)) &= \Phi(0,x) + \frac{\|X_t(x) - x\|^2}{2t}\\
            &= \Phi(0,x) + \frac{t\|M(x) - x\|^2}{2}\\
            &= \Psi(x) - \frac{\|x\|^2}{2} + \frac{t\|\nabla\Psi(x) - x\|^2}{2}
    \end{align*}
    and
    \begin{equation*}
       \nabla X_t(x) = (1-t)\II + t\nabla^2 \Psi(x). 
    \end{equation*}
    Substituting $f(t,x) = \alpha \Phi(t,x)$ and $M(x) = \nabla \Psi(x)$ into (\ref{eq:monge-push-forward}), we get
    \begin{align*}
        &\mu(1,\nabla\Psi(x))Det(\nabla^2\Psi(x)) - \mu(0,x) \\
        &= \int^1_0 \alpha \biggl(\Psi(x) - \frac{\|x\|^2}{2} + \frac{t\|\nabla \Psi(x) - x\|^2}{2}\biggl) Det\Big(t\nabla^2\Psi(x) + (1-t)\II\Big) \d t.
    \end{align*}
\end{proof}
Now, we show the Kantorovich formulation of the problem (\ref{eq:uw2}).
\begin{proposition}[$L^2$ Generalized Unnormalized Kantorovich formulation]
\vspace{0.5cm}
The unnormalized Kantorovich formulation with $f(t,x)$ satisfies
\begin{align*}
    \frac{1}{2}UW_2(\mu_0, \mu_1)^2 
    = \sup_\Phi \biggl\{
        \int_\Omega \biggl( \Phi(1,x) \mu_1(x) - \Phi(0,x) \mu_0(x)\biggl) \d x - \frac{\alpha}{2}   \int^1_0 \int_{\Omega}\Phi(t,x)^2 \d x\d t
    \biggl\},
\end{align*}
where the supremum is taken among all $\Phi:[0,1] \times \Omega \rightarrow \mathbb{R}$ satisfying
\begin{align*}
    \partial_t \Phi(t,x) + \frac{1}{2} \|\nabla \Phi(t,x) \|^2 \leq 0.
\end{align*}
\end{proposition}

\begin{proof}
 We introduce a Lagrange multiplier $\Phi(t,x)$ to reformulate the equation (\ref{eq:UW2-m-mu}).
\begin{align*}
    &\frac{1}{2} UW_2(\mu_0,\mu_1)^2\\
    &= \inf_{\m,\mu,f} \sup_\Phi
        \biggl\{ 
            \int^1_0 \int_\Omega \frac{\|\m(t,x)\|^2}{2\mu(t,x)} + \frac{1}{2\alpha} f(x,t)^2 + \Phi(t,x)\bigl( \partial_t \mu(t,x) + \nabla \cdot \m(t,x) - f(t,x) \bigl) \d x \d t
        \biggl\}\\
    &\geq  \sup_\Phi \inf_{\m,\mu,f}
        \biggl\{ 
            \int^1_0 \int_\Omega \frac{\|\m(t,x)\|^2}{2\mu(t,x)} + \frac{1}{2\alpha} f(x,t)^2 + \Phi(t,x)\bigl( \partial_t \mu(t,x) + \nabla \cdot \m(t,x) - f(t,x) \bigl) \d x \d t
        \biggl\}\\
    &=  \sup_\Phi \inf_{\m,\mu,f}
        \biggl\{ 
            \int^1_0 \int_\Omega \frac{\|\m(t,x)\|^2}{2\mu(t,x)}
            - \nabla \Phi(t,x) \cdot \m(t,x)
            + \frac{1}{2\alpha} f(x,t)^2 
            + \Phi(t,x) \cdot \bigl( \partial_t \mu(t,x) - f(t,x) \bigl) \d x \d t
        \biggl\}\\
    &=  \sup_\Phi \inf_{\m,\mu,f}
        \biggl\{ 
            \int^1_0 \int_\Omega \frac{1}{2}\biggl\|\frac{\m(t,x)}{\mu(t,x)} - \nabla \Phi(t,x)\biggl\|^2 \mu(t,x) - \frac{1}{2}\|\nabla\Phi(t,x)\|^2 \mu(t,x)\d x\d t\\
        & \indent\indent + \int_\Omega \Phi(1,x)\mu_1(x) - \Phi(0,x)\mu_0(x) \d x\\
        & \indent\indent + \int^1_0\int_\Omega -\mu(t,x) \partial_t\Phi(t,x) + \frac{1}{2\alpha} f(t,x)^2 - \Phi(t,x) f(t,x)) \d x\d t
        \biggl\}.
\end{align*}
By the Proposition \ref{prop:l2}, the minimizer $\m$ satisfies $\displaystyle \frac{\m(t,x)}{\mu(t,x)} = \nabla \Phi(t,x)$. Thus, 
\begin{align*}
    &= \sup_\Phi \biggl\{
            \int_\Omega \biggl( \Phi(1,x) \mu_1(x) - \Phi(0,x) \mu_0(x)\biggl) \d x\\
            &\indent\indent + \inf_\mu \int^1_0 \int_\Omega -\mu(t,x) \biggl( \partial_t \Phi(t,x) + \frac{1}{2} \|\nabla\Phi(t,x)\|^2 \biggl) \d x\d t\\
            &\indent\indent + \inf_f \int^1_0 \int_\Omega \frac{1}{2\alpha} f(t,x)^2 - \Phi(t,x) f(t,x) \d x\d t
        \biggl\}\\
    &= \sup_\Phi \biggl\{
        \int_\Omega \biggl( \Phi(1,x) \mu_1(x) - \Phi(0,x) \mu_0(x)\biggl) \d x\\
        &\indent\indent + \inf_\mu \int^1_0 \int_\Omega -\mu(t,x) \biggl( \partial_t \Phi(t,x) + \frac{1}{2} \|\nabla\Phi(t,x)\|^2 \biggl) \d x\d t\\
        &\indent\indent + \inf_f \int^1_0 \int_\Omega \frac{1}{2\alpha}
        \biggl(
            f(t,x) - \alpha \Phi(t,x)
        \biggl)^2
        \d x\d t - \frac{\alpha}{2}   \int^1_0 \int_{\Omega}\Phi(t,x)^2 \d x\d t \biggl\}.
\end{align*}
Again from Proposition \ref{prop:l2}, the minimizer satisfies $f(t,x) = \alpha \Phi(t,x)$. With the assumption $\mu(t,x) \geq 0$ for all $t\in[0,1]$ and $x\in\Omega$, the problem can be written with a constraint.
\begin{align*}
    \frac{1}{2} UW_2(\mu_0,\mu_1)^2 = \sup_\Phi \biggl\{
        \int_\Omega \biggl( \Phi(1,x) \mu_1(x) - \Phi(0,x) \mu_0(x)\biggl) \d x - \frac{\alpha}{2}   \int^1_0 \int_{\Omega}\Phi(t,x)^2 \d x\d t: &\\
        \partial_t \Phi(t,x) + \frac{1}{2} \|\nabla \Phi(t,x) \|^2 \leq 0&
    \biggl\}.
\end{align*}

\noindent We next show that the primal-dual gap is zero. 
\begin{align*}
    &\int^1_0\int_\Omega \frac{\m(t,x)^2}{2\mu(t,x)} + \frac{1}{2\alpha} f(t,x)^2 \d x\d t\\
    &= \int^1_0 \int_\Omega \frac{1}{2}\|\nabla \Phi\|^2\mu(t,x)\d x\d t + \frac{\alpha}{2} \int^1_0 \int_\Omega \Phi(t,x)^2 \d x\d t\\
    &= \int^1_0 \int_\Omega \biggl(
        -\frac{1}{2} \|\nabla \Phi(t,x)\|^2 \mu(t,x) + \|\nabla \Phi(t,x)\|^2 \mu(t,x) + \frac{\alpha}{2} \Phi(t,x)^2
    \biggl) \d x\d t\\
    &= \int^1_0\int_\Omega \partial_t \Phi(t,x) \mu(t,x) + \Phi(t,x) \biggl( - \nabla \cdot \bigl(\mu(t,x)\nabla\Phi(t,x) \bigl)\biggl) 
    + \frac{\alpha}{2} \Phi(t,x)^2 \d x\d t\\
    &= \int_\Omega \Phi(1,x)\mu_1(x) - \Phi(0,x)\mu_0(x) \d x \\
    &\indent\indent -\int^1_0\int_\Omega \Phi(t,x) \biggl( \partial_t \mu(t,x) + \nabla \cdot \bigl(\mu(t,x) \nabla \Phi(t,x)\bigl)\biggl) \d x\d t + \frac{\alpha}{2} \int^1_0 \int_\Omega \Phi(t,x)^2 \d x \d t\\
    &= \int_\Omega \Phi(1,x)\mu_1(x) - \Phi(0,x)\mu_0(x) \d x \\
    &\indent\indent -\int^1_0\int_\Omega \Phi(t,x)f(t,x) \d x\d t + \frac{\alpha}{2} \int^1_0 \int_\Omega \Phi(t,x)^2 \d x \d t.
\end{align*}

\noindent Using $f(t,x) = \alpha \Phi(t,x)$, we get
\begin{align*}
    \frac{1}{2} UW_2(\mu_0,\mu_1)^2= \int_\Omega \Phi(1,x)\mu(1,x) - \Phi(0,x)\mu(0,x) \d x - \frac{\alpha}{2} \int^1_0 \int_\Omega \Phi(t,x)^2 \d x \d t.
\end{align*}
This concludes the proof.
\end{proof} 

\begin{remark}
We note that our results and proofs follow directly from the those used in \cite{gangbo2019unnormalized}. The major difference between  \cite{gangbo2019unnormalized} and our paper is that in the case of spatial independent source function, $f(t)= \frac{\alpha}{|\Omega|}\int \Phi(t,x)dx$, while in the case of spatial dependent source function, $f(t,x)=\alpha\Phi(t,x)$. This difference remains in the corresponding Monge problem and Kantorvich problem. In particular, we obtain a new spatial dependent unnormalized Monge-Amp\`ere equation \eqref{NMA}. 
\end{remark}


\section{Numerical methods}\label{section:numerical-methods}
In this section, we propose a Nesterov accelerated gradient descent method to solve $L^2$ unnormalized OT. In addition, we design a primal-dual hybrid gradient method to solve $L^1$ unnormalized OT.


\subsection{$L^2$ Generalized Unnormalized Wasserstein metric}
In this section, we present a new numerical implementation for $L^2$ unnormalized Wasserstein metric. We obtain a unconstrained version of the problem by plugging the PDE constraint into the objective function. Then the accelerated Nesterov gradient descent method is applied to solve the problem. We show that each iteration involves a simple elliptic equation where fast solvers can be applied. This novel numerical method can also be used in normalized optimal transport and unnormalized optimal transport with a spatially independent source function $f(t)$.

Using Proposition \ref{prop:l2}, we can rewrite the equation (\ref{eq:UW2-m-mu}) as follows:
\begin{equation*}
    \begin{split}
        UW_2(\mu_0,\mu_1)^2
            = \inf_{\Phi,\mu} \Biggl\{  &\int^1_0 \int_\Omega \|\nabla \Phi(t,x)\|_2^2 \mu (t,x) \d x \d t + \alpha \int^1_0\int_\Omega |\Phi(t,x)|^2 \d x\d t:\\
            & \partial_t \mu(t,x) + \nabla \cdot(\mu(t,x) \nabla \Phi(t,x)) = \alpha \Phi(t,x),\\
            &\mu(0,x) = \mu_0(x), \mu(1,x) = \mu_1(x)
            \Biggl\}.
    \end{split}    
\end{equation*}

Define an operator $L_\mu = - \nabla \cdot (\mu \nabla)$. The constraint $\partial_t \mu - L_\mu \Phi = \alpha\Phi$ leads to
\begin{equation}\label{formula:L-operator}
\Phi = (L_\mu + \alpha Id)^{-1} \partial_t \mu.
\end{equation}


\noindent With (\ref{formula:L-operator}), the minimization problem can be reformulated as
\begin{equation}\label{UW-modified-main-prob}
\begin{split}
    UW_2(\mu_0,\mu_1)^2
        = \inf_{\mu} \Biggl\{ & \int^1_0 \int_\Omega \mu \|\nabla (L_\mu + \alpha Id)^{-1} \partial_t \mu\|_2^2 \d x\d t + \alpha \int^1_0\int_\Omega | (L_\mu + \alpha Id)^{-1} \partial_t \mu |^2 \d x\d t:\\
        &\mu(0,x) = \mu_0(x), \mu(1,x) = \mu_1(x)
        \Biggl\}.
    \end{split}
\end{equation}
\noindent Using integration by parts,
\begin{align*}
    &\int^1_0 \int_\Omega \mu \|\nabla (L_\mu + \alpha Id)^{-1} \partial_t \mu\|^2 \d x\d t + \alpha \int^1_0\int_\Omega | (L_\mu + \alpha Id)^{-1} \partial_t \mu |^2 \d x\d t\\
    &= \int^1_0 \int_\Omega 
    - \biggl(\nabla \mu \nabla (L_\mu + \alpha Id)^{-1} \partial_t \mu \biggl)\biggl( (L_\mu + \alpha Id)^{-1} \partial_t \mu \biggl) \d x\d t\\
    &\indent + \alpha \int^1_0\int_\Omega | (L_\mu + \alpha Id)^{-1} \partial_t \mu |^2 \d x\d t\\
    &= \int^1_0 \int_\Omega 
    \biggl(L_\mu (L_\mu + \alpha Id)^{-1} \partial_t \mu \biggl)\biggl( (L_\mu + \alpha Id)^{-1} \partial_t \mu \biggl) \d x\d t\\
    &\indent + \alpha \int^1_0\int_\Omega | (L_\mu + \alpha Id)^{-1} \partial_t \mu |^2 \d x\d t\\
    &= \int^1_0 \int_\Omega 
    \biggl((L_\mu + \alpha Id) (L_\mu + \alpha Id)^{-1} \partial_t \mu \biggl)\biggl( (L_\mu + \alpha Id)^{-1} \partial_t \mu \biggl) \d x\d t\\
    &= \int^1_0 \int_\Omega 
     \partial_t \mu(t,x) (L_\mu + \alpha Id)^{-1} \partial_t \mu(t,x) \d x\d t.
\end{align*}

\noindent Thus, the unnormalized Wasserstein-2 distance forms
\begin{equation}\label{eq:UW-modified-main-prob}
    \begin{split}
        UW_2(\mu_0,\mu_1)^2
            = \inf_{\mu} \Biggl\{ \int^1_0 \int_\Omega 
         \partial_t \mu(t,x) (L_\mu + \alpha Id)^{-1} \partial_t \mu(t,x) \d x\d t:\\
            \mu(0,x) = \mu_0(x), \mu(1,x) = \mu_1(x)
            \Biggl\}.
    \end{split}    
\end{equation}

\begin{proposition}\label{prop:euler-lagrange-of-rho-lap}
If $\mu(t,x)>0$, then the Euler-Lagrange equation of problem (\ref{eq:UW-modified-main-prob}) satisfies the Hamilton-Jacobi equation, i.e.
\begin{align*}
    \partial_t \Phi(t,x) + \frac{1}{2} \|\nabla \Phi(t,x)\|^2 = 0,
    \indent x\in \Omega, t \in [0,1]
\end{align*}
where $\Phi(t,x) = (L_\mu + \alpha Id)^{-1} \partial_t \mu(t,x)$.
\end{proposition}

\begin{remark}
For unnormalized optimal transport with a spatially independent source function $f(t)$, the formula uses $(L_\mu + \frac{\alpha}{|\Omega|} \int_\Omega)^{-1}$ instead of $(L_\mu + \alpha Id)^{-1}$, i.e.
\begin{equation*}
    \begin{split}
       UW_2(\mu_0,\mu_1)^2
            = \inf_{\mu} \Biggl\{ \int^1_0 \int_\Omega 
         \partial_t \mu(t,x) \biggl(L_\mu + \frac{\alpha}{|\Omega|} \int_\Omega\biggl)^{-1} \partial_t \mu(t,x) \d x\d t:\\
            \mu(0,x) = \mu_0(x), \mu(1,x) = \mu_1(x)
            \Biggl\}.
    \end{split}    
\end{equation*}
The Euler-Lagrange equation satisfies the following:
\begin{align*}
    \partial_t \Phi(t,x) + \frac{1}{2}\|\nabla \Phi(t,x)\|^2 = 0,\indent x\in \Omega, t \in [0,1]
\end{align*}
where $\Phi(t,x) = \bigl( L_{\mu} +\frac{\alpha }{|\Omega|}\int_\Omega \bigl)^{-1} \partial_t \mu(t,x)$.
\end{remark}
\begin{remark}
If $\mu(t,x)=0$, one can show that the Euler-Lagrange equation of problem (\ref{eq:UW-modified-main-prob}) satisfies 
\begin{equation*}
        \partial_t \Phi(t,x) + \frac{1}{2} \|\nabla \Phi(t,x)\|^2 \leq 0.
\end{equation*}
\end{remark}
\begin{proof}
    Define 
    $$\ci(\mu) = \int^1_0 \int_\Omega 
         \partial_t \mu(t,x) (L_\mu + \alpha Id)^{-1} \partial_t \mu(t,x) \d x \d t.$$
    We now calculate the first variation of $\cf(\mu)$ with a perturbation $\eta(t,x) \in C^\infty(\Omega \times [0,1])$.
    \begin{align*}
        0 &= \lim_{h\rightarrow 0} \frac{\ci(\mu+ h\eta) - \ci(\mu)}{h}&&\\
            &= \lim_{h\rightarrow 0} \frac{1}{h}\int^1_0 \int_\Omega \left((\partial_t \mu + h \partial_t \eta) (L_{\mu + h\eta} + \alpha Id)^{-1} (\partial_t \mu + h \partial_t \eta) - \partial_t \mu(t,x) (L_\mu + \alpha Id)^{-1} \partial_t \mu(t,x)\right) \d x \d t\\
            &= \lim_{h\rightarrow 0} \Biggl[ \int^1_0\int_\Omega \partial_t \mu \biggl( \frac{(L_{\mu + h\eta} + \alpha Id)^{-1} - (L_{\mu} + \alpha Id)^{-1}}{h} \biggl) \partial_t \mu\\
            & \hspace{8cm} + 2 \partial_t \eta (L_{\mu+h\eta} + \alpha Id)^{-1} \partial_t \mu \d x \d t + O(h)
            \Biggl]\\
            &= \int^1_0\int_\Omega - \partial_t \mu (L_\mu + \alpha Id)^{-1} L_{\eta} (L_\mu + \alpha Id)^{-1} \partial_t \mu + 2 \partial_t \eta (L_{\mu} + \alpha Id)^{-1} \partial_t \mu \d x \d t\\
            &= \int^1_0 \int_\Omega - \Phi L_\eta \Phi + 2\Phi \partial_t \eta  \d x \d t\\
            &=\int^1_0 \int_\Omega - \eta \biggl(\|\nabla \Phi\|^2 + 2 \partial_t \Phi \biggl) \d x \d t.
    \end{align*}
    This has to be true for all $\eta \in C^\infty(\Omega \times [0,1])$. Thus, we get
    \begin{align*}
        \partial_t \Phi + \frac{1}{2} \|\nabla \Phi\|^2 = 0, \indent x\in \Omega, t\in [0,1].
    \end{align*}
    \noindent This concludes the proof.
\end{proof}

Using Proposition \ref{prop:euler-lagrange-of-rho-lap}, we can formulate a Nesterov accelerated  gradient descent method~\cite{nesterov1983method} to solve the minimization problem (\ref{eq:UW-modified-main-prob}).


\begin{algorithm}[H]
\caption{Nesterov Gradient descent method for $UW_2$ with $f(t,x)$} \label{alg:gradient-descent}
    \begin{algorithmic}
        \WHILE{not converged}
            \STATE $\mu^{k+\frac{1}{2}} = \mu^k - \tau \nabla \ci(\mu^k) = \mu^k + \frac{\tau}{2} \bigl( \partial_t \Phi^k + \frac{1}{2}\|\nabla \Phi^k \|^2 \bigl)\indent \text{ where } \Phi^k = (L_{\mu^k} + \alpha Id)^{-1} \partial_t\mu^k$
            \STATE $\mu^{k+\frac{1}{2}}=\max\{\mu^{k+\frac{1}{2}},0\}$
            \STATE $\mu^{k+1} = (1-\gamma^k)\mu^{k+\frac{1}{2}} + \gamma^k \mu^k$
        \ENDWHILE
    \end{algorithmic}
\end{algorithm}
\noindent Here, $\tau$ and $\gamma^k$ are step sizes of the algorithm.
\begin{align*}
    \gamma^k  &= \frac{1-\lambda^k}{\lambda^{k+1}}, \quad\lambda^0 = 0,\quad \lambda^k = \frac{1+\sqrt{1+4(\lambda^{k-1})^2}}{2}  .
\end{align*}

\begin{remark}
    The Nesterov accelerated gradient descent method can be used for a spatially independent source function $f(t)$. We simply replace the operator $L_\mu + \alpha Id$ with $L_\mu + \alpha \int_\Omega$ from Algorithm \ref{alg:gradient-descent}.
\end{remark}

\begin{remark}
  Here we apply an iterative method, such as conjugate gradient, to solve $(L_{\mu^k} + \alpha Id)^{-1} \partial_t\mu^k$.
\end{remark}

\begin{remark}
We remark that variational problem \eqref{eq:UW-modified-main-prob} is convex w.r.t. $\mu(t,x)$. This fact holds following the second variational formula derived in Lemma 2 of \cite{LiG}. Hence our gradient descent algorithm is applied to a convex optimization problem \ref{eq:UW-modified-main-prob}. 
\end{remark}


We next present the discretization of density path in both time and spatial domains, where the spatial domain is given by $1D$ or $2D$. Here we formulate the operator $L_\mu$ and derive its inverse into matrix forms; see similar approaches in \cite{LiG}.

\subsubsection{1D Discretization}
\noindent Consider  the following one dimensional discretization:
\begin{align*}
    &\bmu = (\bmu^0,\cdots,\bmu^{N_t}) \in \RR^{(N_t+1) \times N_x}\\
    &\bmu^n = (\mu^n_0,\cdots,\mu^n_{N_x-1}) \in \RR^{N_x} \indent (n=0,\cdots,N_t)\\
    &\mu^n_i \in \RR \indent (i=0,\cdots,N_x - 1, n=0,\cdots,N_t)\\
    &\mu^0_i = \mu_0(i\dx), \indent \mu^{N_t}_i = \mu_1(i\dx), \indent (i=0,\cdots,N_x - 1)\\
     &\dx = \frac{|\Omega|}{N_x-1}\indent \dt = \frac{1}{N_t}.
\end{align*}
\noindent Using the finite volume method,  the weighted Laplacian operator $\tilde{L}_{\bmu^n,\alpha} := L_{\bmu^n} + \alpha Id$ can be represented as the following matrix:
    \begin{equation*}
        \begin{split}
             &\tilde{L}_{\bmu^n,\alpha}
             = \begin{pmatrix} 
                 \frac{\mu^n_0 + \mu^n_1}{2\dx^2} & -\frac{\mu^n_0 + \mu^n_1}{2\dx^2} & 0 & \cdots & 0\\
                -\frac{\mu^n_0 + \mu^n_1}{2\dx^2} & \frac{\mu^n_0 + \mu^n_1}{2\dx^2} + \frac{\mu^n_1 + \mu^n_2}{2\dx^2} & -\frac{\mu^n_1 + \mu^n_2}{2\dx^2} & \cdots & 0\\
                0 & - \frac{\mu^n_1 + \mu^n_2}{2\dx^2} & \frac{\mu^n_1 + \mu^n_2}{2\dx^2} + \frac{\mu^n_2 + \mu^n_3}{2\dx^2} & \cdots & 0\\
                \vdots & \ddots & \ddots & \ddots & \vdots\\
                0 & \cdots & \cdots & \cdots & - \frac{\mu^n_{N_x-2}+\mu^n_{N_x-1}}{2\dx^2}
                \end{pmatrix}  + \alpha Id
        \end{split}
    \end{equation*}
Further using the forward Euler method in time, formula (\ref{eq:UW-modified-main-prob}) can be discretized as
\begin{align*}
    &\int^1_0 \int_\Omega 
     \partial_t \mu(t,x) (L_\mu + \alpha Id)^{-1} \partial_t \mu(t,x) \d x \d t\\
    &\approx \dt\dx \sum^{N_t-1}_{n=0} 
        \left< \frac{\bmu^{n+1} - \bmu^{n}}{\dt},  (L_{\bmu^n} + \alpha Id)^{-1} \frac{\bmu^{n+1} - \bmu^{n}}{\dt}\right>_{L^2}\\
    &=\frac{\dx}{\dt} \sum^{N_t-1}_{n=0} 
         \left<\bmu^{n+1} - \bmu^{n}, (L_{\bmu^n} + \alpha Id)^{-1} (\bmu^{n+1} - \bmu^{n})\right>_{L^2}
\end{align*}
with $\bmu^0$ and $\bmu^{N_t}$ are given. $\langle \cdot, \cdot \rangle_{L^2}$ is $L^2$ norm in $\RR^{N_x}$ such that
\begin{align*}
    \langle \bm{a},\bm{b}\rangle_{L^2} = \sum^{N_x-1}_{i=0} a_i b_i \indent\text{ for } \bm{a},\bm{b} \in \RR^{N_x} .
\end{align*}
   
We are now ready to present the derivative of $E(\bmu)$, and formulate the discrete Hamilton-Jacobi equation as in Algorithm \ref{alg:gradient-descent}.
\begin{proposition}\label{prop:1d-discretization-EL}
Denote $\tilde{L}_{\bmu^n,\alpha} := L_{\bmu^n} + \alpha Id$. Let 
$$E(\bmu) :=\frac{\dx}{\dt} \sum^{N_t-1}_{n=0} 
         \left<\bmu^{n+1} - \bmu^{n}, \tilde{L}_{\bmu^n,\alpha}^{-1} (\bmu^{n+1} - \bmu^{n})\right>_{L^2}.$$
Suppose $x \in \Omega$. The derivative of $E(\bmu)$ with respect to $\bmu^n$ ($n=1,\cdots, N_t-1$) is
\begin{align*}
    \frac{\delta E(\bmu)}{\delta \bmu^n}
    &= \frac{\dx}{\dt}\Biggl(
        - 2 \tilde{L}_{\bmu^n,\alpha}(\bmu^{n+1} - \bmu^{n})
        + 2 \tilde{L}_{\bmu^{n-1},\alpha}(\bmu^{n} - \bmu^{n-1})\\\nonumber
        &- \biggl(\left< \tilde{L}_{\bmu^n,\alpha}^{-1}(\bmu^{n+1} - \bmu^n),  L_{\bm{e}_i} \tilde{L}_{\bmu^n,\alpha}^{-1}(\bmu^{n+1} - \bmu^n) \right>_{L^2}\biggl)^{N_x-1}_{i=0}
    \Biggl)
\end{align*}
where $\bm{e}_i \in \RR^{N_x}$ is an index vector defined as
\begin{equation*}
    \bm{e}_i = 
    \begin{cases}
        1 & \text{$i^{th}$ index}\\
        0 & \text{else}.
    \end{cases}
\end{equation*}

\end{proposition}

\begin{proof}
    Differentiating $E(\bmu)$ with respect to $\bmu^n$ for $n=1, \cdots, N_t-1$, we get
    \begin{align*}
        \frac{\delta E(\bmu)}{\delta \bmu^n}
        &= \frac{\delta}{\delta \bmu^n} \Biggl(
        \dt \dx \sum^{N_t-1}_{m=0} (\bmu^{m+1} - \bmu^{m}) \tilde{L}_{\bmu^n,\alpha}^{-1}(\bmu^{m+1} - \bmu^{m})
        \Biggl)\\
        &= \dt \dx \Biggl(
            - 2 \tilde{L}_{\bmu^n,\alpha}^{-1}(\bmu^{n+1} - \bmu^n)
            + 2 \tilde{L}_{\bmu^{n-1},\alpha}^{-1}(\bmu^{n} - \bmu^{n-1})\\
            &\hspace{2cm} + (\bmu^{n+1} - \bmu^n)\frac{\partial \tilde{L}_{\bmu^n,\alpha}^{-1}}{\partial \bmu^n}(\bmu^{n+1} - \bmu^n)
        \Biggl),
    \end{align*}
    \noindent and
    \begin{align*}
        (\bmu^{n+1} - \bmu^n)\frac{\delta \tilde{L}_{\bmu^n,\alpha}^{-1}}{\delta \bmu^n_n}(\bmu^{n+1} - \bmu^n) 
        &= - \left<\bmu^{n+1} - \bmu^n, \tilde{L}_{\bmu^n,\alpha}^{-1} L_{\bm{e_i}} \tilde{L}_{\bmu^n,\alpha}^{-1}(\bmu^{n+1} - \bmu^n) \right>_{L^2}\\ 
        &= - \left< \tilde{L}_{\bmu^n,\alpha}^{-1}(\bmu^{n+1} - \bmu^n),  L_{\bm{e}_i} \tilde{L}_{\bmu^n,\alpha}^{-1}(\bmu^{n+1} - \bmu^n) \right>_{L^2}.
    \end{align*}
    This concludes the proof.
\end{proof}

    \noindent Consider $\bm{u} = (u_0, \cdots, u_{N_x-1})^T \in \RR^{N_x}$, then $\left<\bm{u}, L_{\bm{e}_i} \bm{u}\right>_{L_2}$ forms the R.H.S. of the discrete Hamilton-Jacobi equation as follows
    \begin{equation*}
        \left< \bm{u}, L_{\bm{e}_i} \bm{u} \right>_{L^2} =
        \begin{cases}
         &\frac{1}{2}\left(\frac{u_{i+1} - u_{i}}{\dx}\right)^2 + \frac{1}{2}\left(\frac{u_{i} - u_{i-1}}{\dx}\right)^2, \indent i = 1,\cdots,N_x-2\\
         &\frac{1}{2}\left(\frac{u_{i+1} - u_{i}}{\dx}\right)^2 , \indent i = 0\\
         &\frac{1}{2}\left(\frac{u_{i} - u_{i-1}}{\dx}\right)^2, \indent i = N_x-1.
         \end{cases}
    \end{equation*}
    
\subsubsection{2D Discretization} Now, consider the two dimensional discretization. Assume $\Omega = [0,1] \times [0,1]$ and $t\in[0,1]$.
\begin{align*}
    &\bmu = (\bmu^0,\cdots,\bmu^{N_t}) \in \RR^{(N_t+1) \times N_x \times N_y}\\
    &\bmu^n = (\mu^n_{ij})_{i=0}^{N_x-1}{}_{j=0}^{N_y-1} \in \RR^{N_x \times N_y} \indent (n=0,\cdots,N_t)\\
    &\mu^n_{ij} \in \RR \indent (i=0,\cdots,N_x - 1,j=0,\cdots,N_y - 1, n=0,\cdots,N_t)\\
    &\mu^0_{ij} = \mu_0(i\dx, j\dy), \indent \mu^{N_t}_{ij} = \mu_1(i\dx, j\dy), \indent (i=0,\cdots,N_x - 1, j =0, \cdots,N_y - 1)\\
     &\dx = \frac{1}{N_x-1}, \indent \dy = \frac{1}{N_y-1}, \indent \dt = \frac{1}{N_t}.
\end{align*}
\noindent Similar to 1D case, using the finite volume method, formula (\ref{eq:UW-modified-main-prob}) can be discretized as
\begin{align*}
    &\int^1_0 \int^1_0 \int^1_0 
     \partial_t \mu(t,x,y) (L_\mu + \alpha Id)^{-1} \partial_t \mu(t,x,y) \d x \d y \d t\\
    &\approx \frac{\dx\dy}{\dt} \sum^{N_t-1}_{n=0} 
         \left<\bmu^{n+1} - \bmu^{n}, (L_{\bmu^n} + \alpha Id)^{-1} (\bmu^{n+1} - \bmu^{n})\right>_{L^2}
\end{align*}
with $\bmu^0$ and $\bmu^{N_t}$ are given and $\langle \cdot, \cdot \rangle_{L^2}$ is defined as
\begin{align*}
    \langle \bm{a},\bm{b}\rangle_{L^2} = \sum^{N_x-1}_{i=0}\sum^{N_y-1}_{j=0} a_{ij} b_{ij} \indent\text{ for } \bm{a},\bm{b} \in \RR^{N_x \times N_y} .
\end{align*}

\noindent The major difference between 1D discretization and 2D discretization arises from the weighted Laplacian operator $\tilde{L}_{\bmu^n,\alpha}$. Consider $\bm{w} = (w_{i,j})_{i=0}^{N_x-1}{}_{j=0}^{N_y-1} \in \RR^{N_x \times N_y}$. For $i=0,\cdots,N_x-1$ and $j=0,\cdots,N_y-1$, the operator can be described as follows:
    \begin{align*}
        &(\tilde{L}_{\bmu^n,\alpha} \bm{w})_{ij}\\
        =&  - \frac{1}{\dx^2} \Biggl(
            \frac{\mu^n_{i+1,j} + \mu^n_{i,j}}{2}w_{i+1,j} - 2 \biggl( \frac{\mu^n_{i+1,j} + \mu^n_{i,j}}{2} + \frac{\mu^n_{i,j} + \mu^n_{i-1,j}}{2}\biggl)w_{i,j} + \frac{\mu^n_{i,j} + \mu^n_{i-1,j}}{2} w_{i-1,j}
        \Biggl)\\
        &- \frac{1}{\dy^2} \Biggl(
            \frac{\mu^n_{i,j+1} + \mu^n_{i,j}}{2}w_{i,j+1} - 2 \biggl( \frac{\mu^n_{i,j+1} + \mu^n_{i,j}}{2} + \frac{\mu^n_{i,j} + \mu^n_{i,j-1}}{2}\biggl)w_{i,j} + \frac{\mu^n_{i,j} + \mu^n_{i,j-1}}{2} w_{i,j-1}
        \Biggl)\\
        & + \alpha w_{i,j}.
    \end{align*}
    Here, we assume the Neumann boundary on the spatial domain $\Omega$. Thus,
    \begin{align*}
        w_{-1,j} = w_{0,j},\indent w_{N_x,j} = w_{N_x-1,j},\indent j=0,\cdots,N_y-1\\
        w_{i,-1} = w_{i,0},\indent w_{i,N_y} = w_{i,N_y-1},\indent i=0,\cdots,N_x-1\\
        \mu^n_{-1,j} = \mu^n_{0,j},\indent \mu^n_{N_x,j} = \mu^n_{N_x-1,j},\indent j=0,\cdots,N_y-1\\
        \mu^n_{i,-1} = \mu^n_{i,0},\indent \mu^n_{i,N_y} = \mu^n_{i,N_y-1},\indent i=0,\cdots,N_x-1.
    \end{align*}
    
\begin{proposition}
Denote $\tilde{L}_{\bmu^n,\alpha} := L_{\bmu^n} + \alpha Id$. Let 
$$E(\bmu) :=\frac{\dx \dy}{\dt} \sum^{N_t-1}_{n=0} 
         \left<\bmu^{n+1} - \bmu^{n}, \tilde{L}_{\bmu^n,\alpha}^{-1} (\bmu^{n+1} - \bmu^{n})\right>_{L^2}.$$
Suppose $x \in \Omega = [0,1] \times[0,1]$. The derivative of $E(\bmu)$ with respect to $\bmu^n$ ($n=1,\cdots, N_t-1$) is
\begin{align*}
    \frac{\delta E(\bmu)}{\delta \bmu^n}
    &= \frac{\dx \dy}{\dt}\Biggl(
        - 2 \tilde{L}_{\bmu^n,\alpha}(\bmu^{n+1} - \bmu^{n})
        + 2 \tilde{L}_{\bmu^{n-1},\alpha}(\bmu^{n} - \bmu^{n-1})\\\nonumber
        &- \biggl(\left< \tilde{L}_{\bmu^n,\alpha}^{-1}(\bmu^{n+1} - \bmu^n),  L_{\bm{e}_{ij}} \tilde{L}_{\bmu^n,\alpha}^{-1}(\bmu^{n+1} - \bmu^n) \right>_{L^2}\biggl)^{N_x-1, N_y-1}_{i=0, j=0}
    \Biggl).
\end{align*}
where $\bm{e}_{ij}$ is an index vector such that $\bm{e}_{k,l} = 1$ if $k=i$ and $l=j$ and $0$ otherwise.
\end{proposition}

\begin{proof}
    The proof follows exactly the one in proposition \ref{prop:1d-discretization-EL}.
\end{proof}

    \noindent Consider a vector $\bm{u} = (u_{ij})^{N_x-1}_{i=0}{}^{N_y-1}_{j=0} \in \RR^{N_x \times N_y}$ that satisfies the Neumann boundary condition. Similar to 1D case, $\left< \bm{u}, L_{\bm{e}_{i,j}} \bm{u} \right>_{L^2}$ can be computed easily based on the operator and it forms the R.H.S. of the discrete Hamilton-Jacobi equation. For $i=0,\cdots,N_x-1$ and $j=0,\cdots,N_y-1$,
    \begin{equation*}
    \begin{split}
        \left< \bm{u}, L_{\bm{e}_{i,j}} \bm{u} \right>_{L^2} =&
        \frac{1}{2}\left(\frac{u_{i+1,j} - u_{i,j}}{\dx}\right)^2 + \frac{1}{2}\left(\frac{u_{i,j} - u_{i-1,j}}{\dx}\right)^2\\
        +&
        \frac{1}{2}\left(\frac{u_{i,j+1} - u_{i,j}}{\dy}\right)^2 + \frac{1}{2}\left(\frac{u_{i,j} - u_{i,j-1}}{\dy}\right)^2.
    \end{split}
    \end{equation*}
    

\subsection{$L^1$ Generalized Unnormalized Wasserstein metric}
Our discussion here mainly focuses on $\|u\|_1 = \sum_{i} |u_i|$. The algorithm can be simply extended to $\|u\|_2 = \sqrt{\sum_i u_i^2}$ using the corresponding shrinkage operator.  With the Lagrangian (\ref{prop:l1-lagrangian}), we consider a saddle point problem.
\begin{align*}
    \inf_{\m,c} \sup_{\Phi} \cl (m,c,\Phi).
\end{align*}
We can use PDHG~\cite{CP} to solve the saddle point problem by minimizing $\cl(\m,c,\Phi)$ over $m$ and $c$ and maximizing over $\Phi$.
\begin{align}
    \m^{k+1} &= \argmin_{\m} \biggl( \|\m\|_{1} + \frac{\epsilon}{2} \|\m\|^2_2 + \left< \Phi^k, \nabla \cdot \m \right>_{L^2} +\frac{1}{2\lambda} \|\m - \m^k\|^2_2 \biggl)\label{eq:L1-mk_plus_1}\\
    c^{k+1} &= \argmin_{c} \biggl( \frac{1}{\alpha} \|c\|_{1}  + \frac{\epsilon}{2}\|c\|^2_2 - \left< \Phi^k, c \right>_{L^2}  +\frac{1}{2\lambda} \|c -c^k\|^2_2  \biggl)\label{eq:L1-ck_plus_1}\\
    \Phi^{k+1} &= \argmax_\Phi \biggl( \big< \Phi, \nabla \cdot ( 2\m^{k+1} - \m^k) - (2 c^{k+1} - c^k) + \mu_1 - \mu_0 \big>_{L^2} - \frac{1}{2 \tau} \|\Phi - \Phi^k\|^2_2 \biggl)
\end{align}
where $\lambda$ and $\tau$ are step sizes of the algorithm. 
Note that we add a small $\|\cdot\|_2^2$ perturbation in \eqref{eq:L1-mk_plus_1} and \eqref{eq:L1-ck_plus_1} to strictly convexify the problem. This adjustment can overcome the possible non-uniqueness of the optimal transport problem. This trick is also related to so called the elastic net regularization~\cite{parikh2014proximal},
whose proximal operator is essentially the same as the proximal operator of $L^1$ norm shrink operator.
\begin{algorithm}[H]
\caption{PDHG for $UW_1$ with $f(t,x)$}\label{alg:L1-pdhg}
\begin{algorithmic}
\STATE $\m^{k+1} = 1/(1 + \epsilon \lambda) shrink\biggl( \m^k + \lambda \nabla \Phi^k, \lambda \biggl)$\\
\STATE $c^{k+1} = 1/(1+\epsilon \lambda) shrink\biggl( c^k + \lambda \Phi^k, \frac{\lambda}{\alpha} \biggl)$\\
\STATE $\Phi^{k+1} = \Phi^k + \tau\biggl( \nabla \cdot(2\m^{k+1} - \m^k) - (2c^{k+1} - c^k) + \mu_1 - \mu_0 \biggl)$
\end{algorithmic}
\end{algorithm}
\noindent where the $shrink$ operator is defined as following:
\[
    (shrink(u,t))_i = 
    \begin{cases}
    (1 - t/|u_i|)u_i, &\text{for } \|u_i\|_1 \geq t;\\
    0, &\text{for } \|u_i\|_1 < t.
    \end{cases} \quad i = 1,\cdots, d.
\]

\begin{remark}
    This algorithm can also be extended to $\|\cdot\|_2$ by simply replacing the above shrink operator as \[
    shrink(u,t) = 
    \begin{cases}
    (1 - t/\|u\|_2)u, &\text{for } \|u\|_2 \geq t;\\
    0, &\text{for } \|u\|_2 < t.
    \end{cases}
\]
\end{remark}



\subsubsection{Discretization}

\noindent Consider  the following two dimensional discretization on a domain $\Omega = [0,1]\times[0,1]$ based on the finite volume method.
\begin{align*}
    &\dx = \frac{1}{N_x}, \dy = \frac{1}{N_y}\\
    &\mu^0_{ij} = \mu_0(i\dx,j\dy),\indent \mu^1_{ij} = \mu_1(i\dx,j\dy)\\
    &V = \{(i,j): i=0,\cdots,N_x, j=0,\cdots,N_y\}\\
    &E_x = \{(i \pm \frac{1}{2},j): i=1,\cdots,N_x-1,j=0,\cdots,N_y)\}\\
    &E_y = \{(i ,j \pm \frac{1}{2}): i=0,\cdots,N_x,j=1,\cdots,N_y-1)\}\\
    &\bm{\Phi} = (\Phi_{ij})_{ij \in V} \in \RR^{(N_x+1)\times (N_y+1)},
    \indent \bm{c} = (c_{ij})_{ij \in V} \in \RR^{(N_x+1)\times (N_y+1)}\\
    &\bm{mx} = (mx_{e})_{e \in E_x} \in \RR^{N_x \times (N_y+1)},
    \indent \bm{my} = (my_{e})_{e \in E_y} \in \RR^{(N_x+1)\times N_y}\\
    &mx_{i+\frac{1}{2},j} \approx \int^{(i+1)\dx}_{i\dx} \int^{(j+1/2)\dy}_{(j-1/2)\dy} m_x(x,y) \d y \d x\\
    &my_{i,j+\frac{1}{2}} \approx \int^{(i+1/2)\dx}_{(i-1/2)\dx} \int^{(j+1)\dy}_{j\dy} m_y(x,y) \d y \d x.
\end{align*}

\noindent Here $m$ satisfies the zero flux condition. Thus, $\bm{mx}$ and $\bm{my}$ satisfy the following boundary conditions on $m$:
\begin{align*}
    &mx_{-\frac{1}{2},j} =  mx_{N_x+\frac{1}{2},j} = 0,\indent j=0,\cdots,N_y\\
    &my_{i, -\frac{1}{2}} = my_{i, N_y+\frac{1}{2}} = 0,\indent i=0,\cdots,N_x.
\end{align*}
The discretization of Algorithm \ref{alg:L1-pdhg} can be written as the following:
\begin{align*}
    mx^{k+\frac{1}{2}}_{i+\frac{1}{2},j} =& \frac{1}{1+\epsilon \lambda} \biggl( mx^k_{i+\frac{1}{2},j} +  \frac{\lambda}{\dx}(\Phi_{i+1,j} - \Phi_{i,j}) \biggl)\\
    my^{k+\frac{1}{2}}_{i,j+\frac{1}{2}} =& \frac{1}{1+\epsilon \lambda} \biggl( my^k_{i,j+\frac{1}{2}} +  \frac{\lambda}{\dy}(\Phi_{i,j+1} - \Phi_{i,j}) \biggl)\\
    c^{k+\frac{1}{2}}_{ij} =& \frac{1}{1+\epsilon \lambda} shrink\biggl( c^k + \lambda \Phi^k_{ij}, \frac{\lambda}{\alpha} \biggl)\\
    \bm{mx}^{k+1} =& 2\bm{mx}^{k+\frac{1}{2}} - \bm{mx}^{k}\\
    \bm{my}^{k+1} =& 2\bm{my}^{k+\frac{1}{2}} - \bm{my}^{k}\\
    \bm{c}^{k+1} =& 2\bm{c}^{k+\frac{1}{2}} - \bm{c}^{k}\\
    \Phi^{k+1}_{ij} =& \Phi^{k}_{ij} + \tau \biggl(
        \frac{1}{\dx} (mx^{k+1}_{i+\frac{1}{2},j}-mx^{k+1}_{i-\frac{1}{2},j})
        + \frac{1}{\dy} (my^{k+1}_{i,j+\frac{1}{2}}-my^{k+1}_{i,j-\frac{1}{2}})
        - c^{k+1}_{ij} + \mu^1_{ij} - \mu^0_{ij}
    \biggl).
\end{align*}




\section{Numerical experiments}\label{section:numerical-experiments}
In this section, we show the numerical results with various examples for $L^1$ and $L^2$ unnormalized optimal transport with the spatially dependent source function.


\subsection{Nesterov Accelerated Gradient Descent for $UW_2$}

We present four numerical experiments with different $\mu_0$ and $\mu_1$ using Algorithm \ref{alg:gradient-descent}.  
\subsubsection{Experiment 1}
Consider a one dimensional problem on $\Omega =  [0,1]$ with $\mu_0$ and $\mu_1$ in $\cm(\Omega)
$ as
\begin{align*}
    \mu_0 &= N(x;\frac{1}{5}, 0.0001)\\
    \mu_1 &=  N(x;\frac{4}{5}, 0.0001) \cdot 1.4
\end{align*}
Here we choose $N(x,\mu,\sigma^2) = C \exp\left(-\frac{(x-\mu)^2}{2\sigma^2}\right)$ with an appropriate choice of $C$ satisfying $ \int_\Omega N(x;\mu,\sigma^2) \d x = 1$. Note that $\int_\Omega \mu_0 \d x = 1$ and $\int_\Omega \mu_1 \d x = 1.4$. We use the Algorithm \ref{alg:gradient-descent} to compute the minimizer $\mu(t,x)$ of $UW_2(\mu_0,\mu_1)$. The parameters chosen for the experiment are $N_x = 40, N_t = 30, \tau = 0.1$.

\begin{figure}[ht]
\begin{minipage}{0.32\linewidth}
\includegraphics[width=1\linewidth]{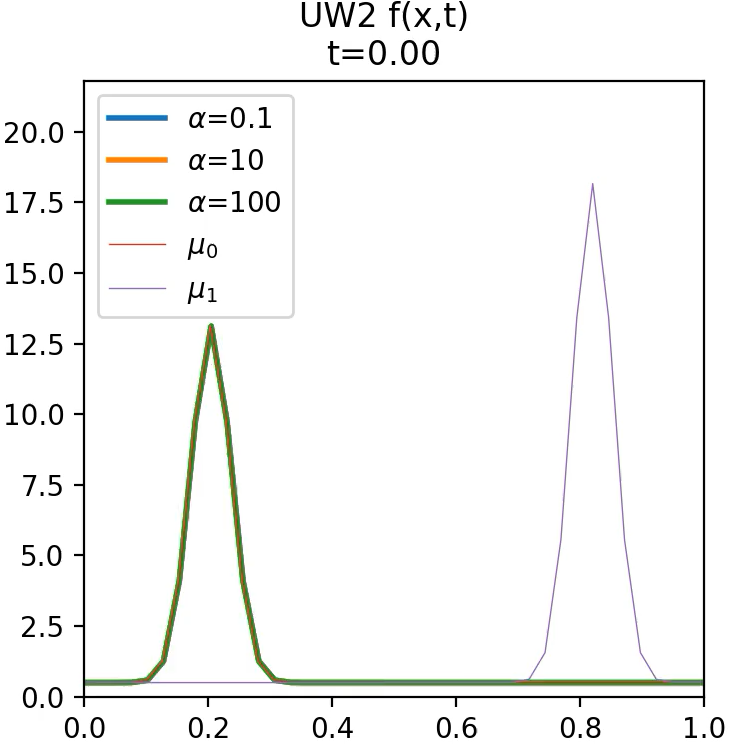}
\end{minipage}\hfill
\begin{minipage}{0.32\linewidth}
\includegraphics[width=1\linewidth]{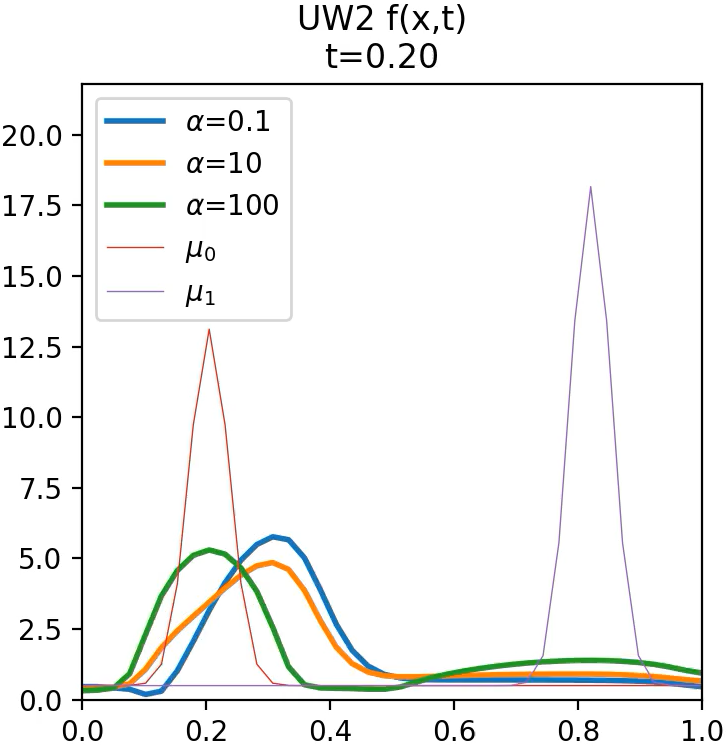}
\end{minipage}\hfill
\begin{minipage}{0.32\linewidth}
\includegraphics[width=1\linewidth]{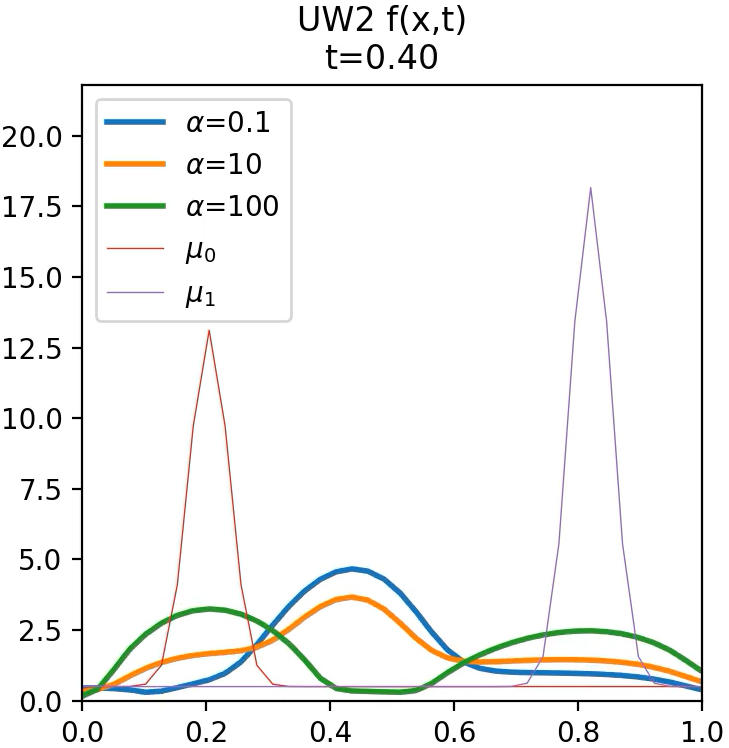}
\end{minipage}\hfill
\begin{minipage}{0.32\linewidth}
\includegraphics[width=1\linewidth]{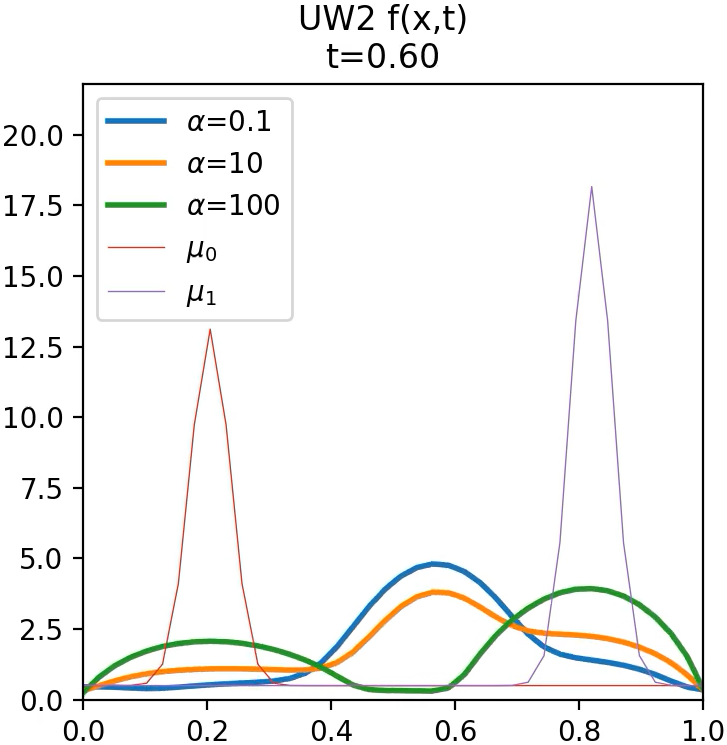}
\end{minipage}\hfill
\begin{minipage}{0.32\linewidth}
\includegraphics[width=1\linewidth]{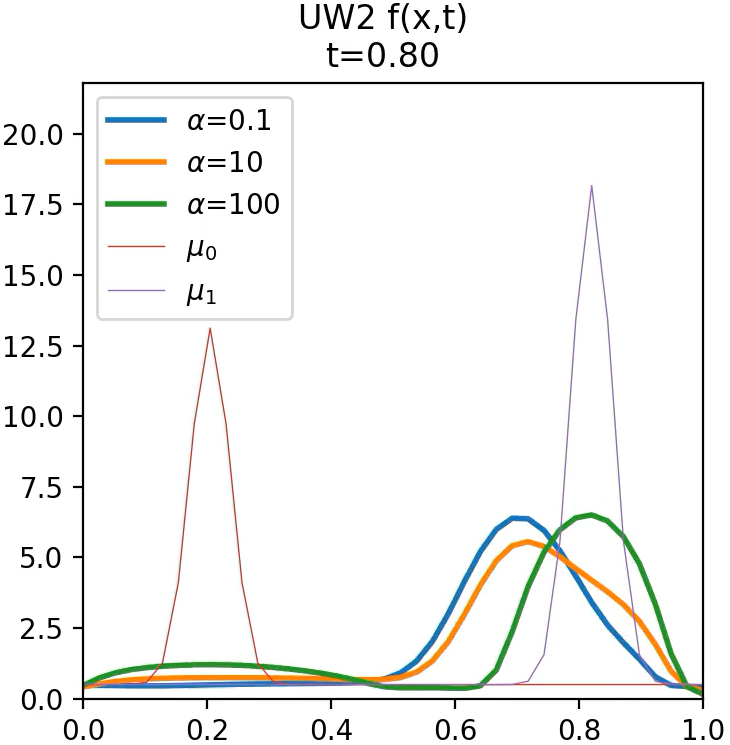}
\end{minipage}\hfill
\begin{minipage}{0.32\linewidth}
\includegraphics[width=1\linewidth]{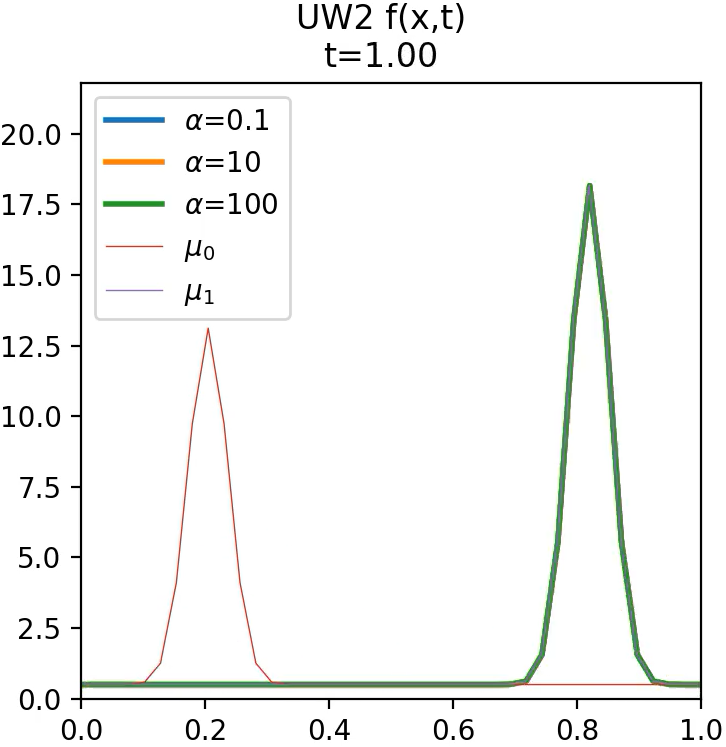}
\end{minipage}\hfill
\caption{\textit{Experiment 1.} $L^2$ Unnormalized optimal transportation with $f(t,x)$ using different $\alpha$ values. Blue line shows $\alpha=0.1$, orange line shows $\alpha=10$, and green line shows $\alpha=100$.}
\label{fig:UW2-1d-ex}
\end{figure}

Figure \ref{fig:UW2-1d-ex} shows the $L^2$ unnormalized optimal transport with a spatially dependent source function $f(t,x)$ with different $\alpha$ values. 
The parameter $\alpha$ determines the ratio between transportation and linear interpolation for $\mu_0$ and $\mu_1$. If $\alpha$ is small, the geodesic of generalized unnormalized optimal transport is similar to the normalized (classical) optimal transport geodesics. As the parameter $\alpha$ increases,
the generalized unnormalized optimal transport geodesic behaves closer to the Euclidean geodesics.

\begin{figure}[ht]
\begin{minipage}{0.32\linewidth}
\includegraphics[width=1\linewidth]{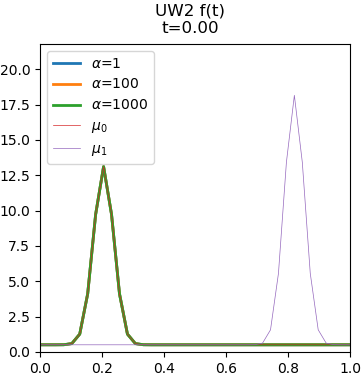}
\end{minipage}\hfill
\begin{minipage}{0.32\linewidth}
\includegraphics[width=1\linewidth]{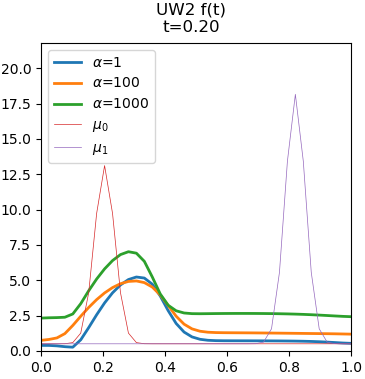}
\end{minipage}\hfill
\begin{minipage}{0.32\linewidth}
\includegraphics[width=1\linewidth]{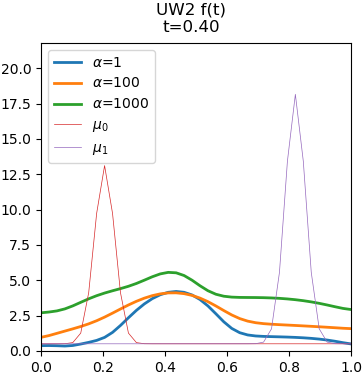}
\end{minipage}\hfill
\begin{minipage}{0.32\linewidth}
\includegraphics[width=1\linewidth]{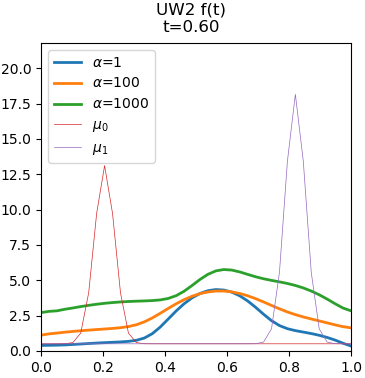}
\end{minipage}\hfill
\begin{minipage}{0.32\linewidth}
\includegraphics[width=1\linewidth]{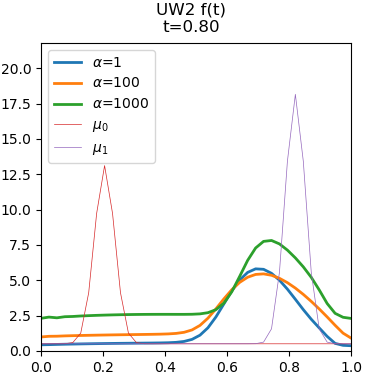}
\end{minipage}\hfill
\begin{minipage}{0.32\linewidth}
\includegraphics[width=1\linewidth]{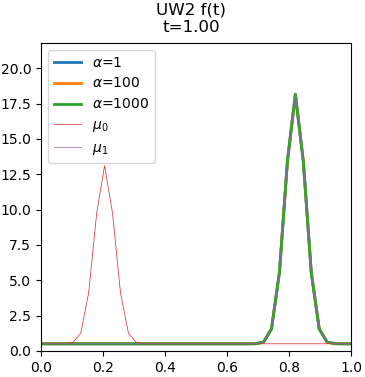}
\end{minipage}\hfill
\caption{\textit{Experiment 1.} $L^2$ Unnormalized optimal transportation with $f(t)$ using different $\alpha$ values. Blue lines show $\alpha=1$, orange lines show $\alpha=100$, and green lines show $\alpha=1000$.}
\label{fig:UW2-1d-ex-ft}
\end{figure}

Figure \ref{fig:UW2-1d-ex-ft} shows the transportation with a spatially independent source function $f(t)$. It is clear to see  that the masses are created or removed locally for the transportation with $f(t,x)$, while they are created or removed globally for the transportation with $f(t)$.  
 
\begin{figure}[ht]
\begin{minipage}[t]{0.45\linewidth}
\includegraphics[width=1\linewidth]{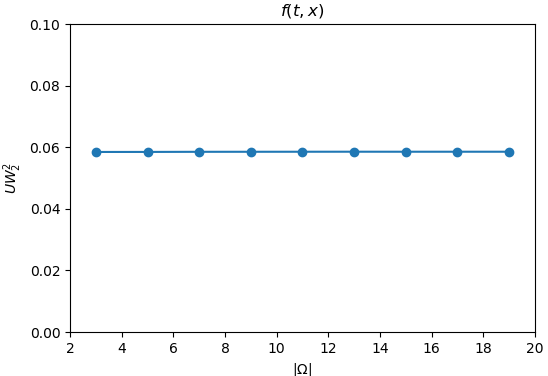}
\caption*{(a) $|\Omega|$ vs. $UW_2^2$ with $f(t,x)$}
\end{minipage}\hfill
\begin{minipage}[t]{0.45\linewidth}
\includegraphics[width=1\linewidth]{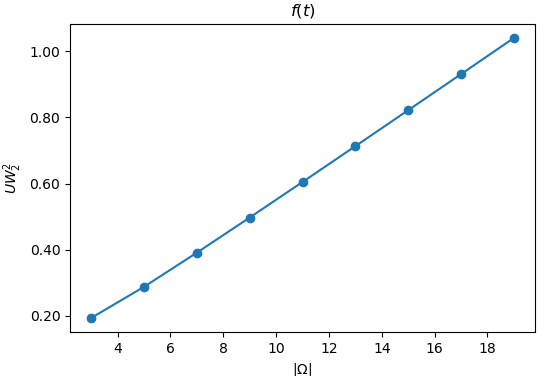}
\caption*{(b) $|\Omega|$ vs. $UW_2^2$ with $f(t)$}
\end{minipage}\hfill
\caption{\textit{Experiment 2.} The size of the domain $|\Omega|$ vs. $L^2$ unnormalized Wasserstein metrics for $f(t,x)$ and $f(t)$. x-axis represents $|\Omega|$ and y-axis represents $UW_2(\mu_0,\mu_1)^2$. Both $f(t,x)$ and $f(t)$ use $\alpha=100$.}
\label{fig:UW2-distance}
\end{figure}

\subsubsection{Experiment 2} In this experiment, we can see how the size of the domain affects the unnormalized Wasserstein distances for both a spatially dependent source function $f(t,x)$ and a spatially independent source function $f(t)$. Consider a one dimensional problem between two densities with different total masses. 
Figure \ref{fig:UW2-distance} shows plots for the size of the domain $|\Omega|$ vs. the unnormalized Wasserstein distance $UW_2$. As expected, for the spatially independent source function, the distance increases as $|\Omega|$ increases since the source function affects the transportation globally. Thus, more masses are created or removed as $|\Omega|$ increases. However, the unnormalized Wasserstein distance with the spatially dependent source function does not depend on $|\Omega|$. This actually provides an advantage of using the spatially dependent source function over the spatially independent source function when we need a consistent Wasserstein distance for any size of the domain.

\begin{figure}[ht]
\begin{minipage}{0.32\linewidth}
\includegraphics[width=1\linewidth]{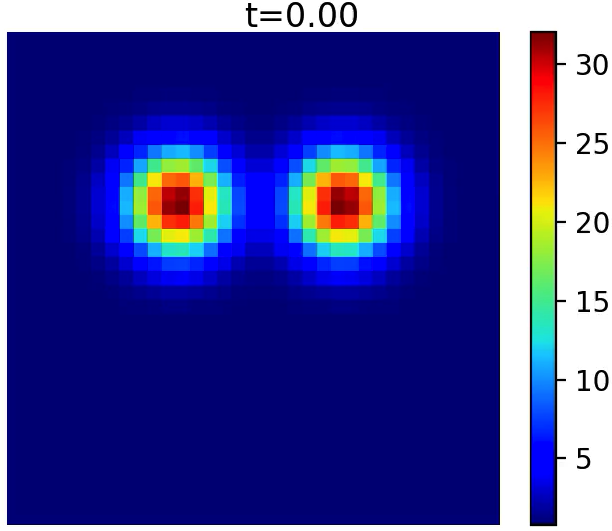}
\end{minipage}\hfill
\begin{minipage}{0.32\linewidth}
\includegraphics[width=1\linewidth]{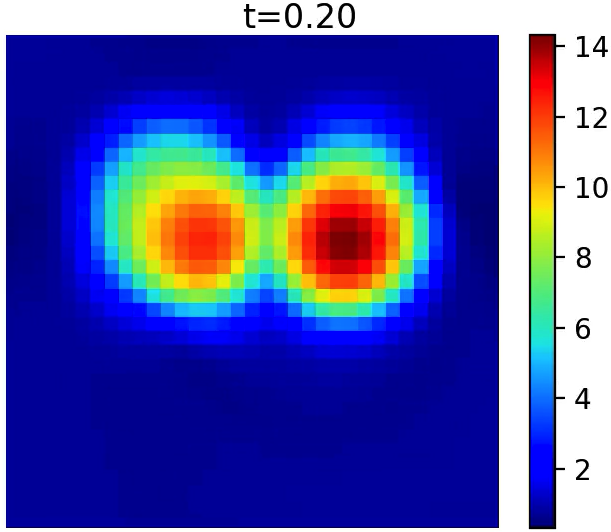}
\end{minipage}\hfill
\begin{minipage}{0.32\linewidth}
\includegraphics[width=1\linewidth]{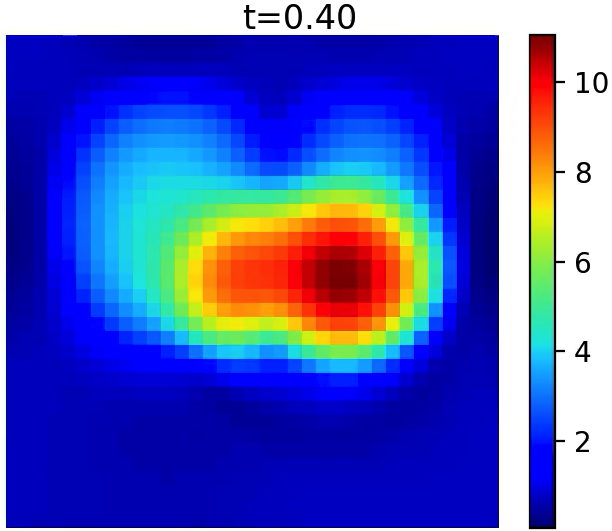}
\end{minipage}\hfill
\begin{minipage}{0.32\linewidth}
\includegraphics[width=1\linewidth]{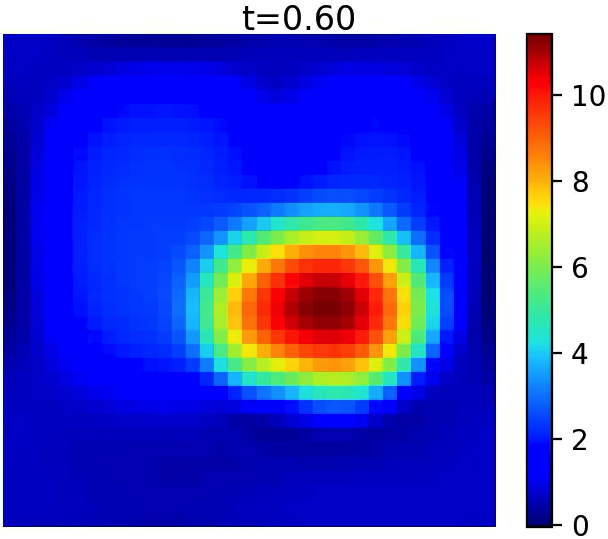}
\end{minipage}\hfill
\begin{minipage}{0.32\linewidth}
\includegraphics[width=1\linewidth]{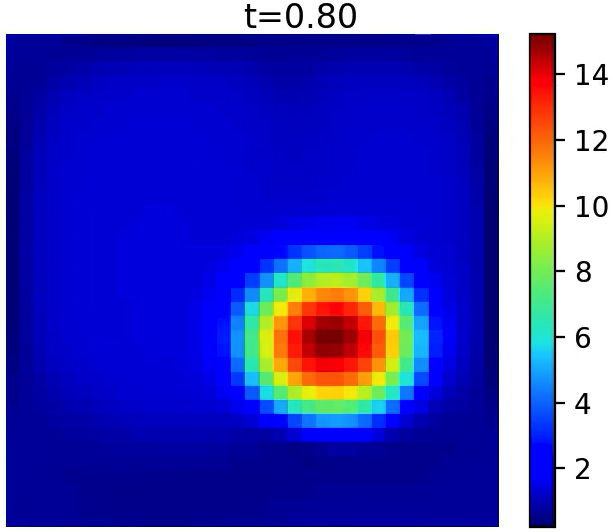}
\end{minipage}\hfill
\begin{minipage}{0.32\linewidth}
\includegraphics[width=1\linewidth]{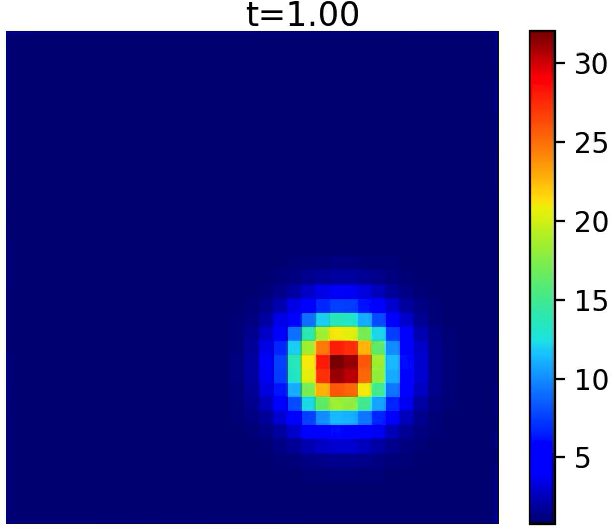}
\end{minipage}\hfill
\caption{\textit{Experiment 3.} $L^2$ generalized unnormalized optimal transportation: 2D example with a spatially dependent source function $f(t,x)$ and $\alpha=1$.}
\label{fig:UW2-2d-ex-alpha-1}
\end{figure}

\subsubsection{Experiment 3}
Consider a two dimensional problem with the following input values:
\begin{align*}
    &\mu_0 = N\left((x,y), (\frac{1}{3},\frac{1}{3}), (\frac{\sqrt{2}}{20},\frac{\sqrt{2}}{20})\right) + N\left((x,y), (\frac{2}{3},\frac{1}{3}), (\frac{\sqrt{2}}{20},\frac{\sqrt{2}}{20})\right)\\
    &\mu_1 = N\left((x,y), (\frac{2}{3},\frac{2}{3}), (\frac{\sqrt{2}}{20},\frac{\sqrt{2}}{20})\right)
\end{align*}
where $N\left((x,y); (\mu_x,\mu_y), (\sigma_x^2,\sigma_y^2)\right) = C \exp\left(-\frac{(x-\mu_x)^2}{2\sigma_x^2}-\frac{(y-\mu_y)^2}{2\sigma_y^2}\right)$ and $C$ is a constant such that $\int_\Omega N((x,y);(\mu_x,\mu_y),(\sigma_x^2,\sigma_y^2)) \d x \d y = 1$. Using the Algorithm \ref{alg:gradient-descent}, we calculate the minimizers of $UW_2(\mu_0,\mu_1)$ with a spatially dependent source function $f(t,x)$. The parameters are chosen as 
$N_x = 35, N_y = 35, N_t = 15, \tau = 0.1$.
Figure \ref{fig:UW2-2d-ex-alpha-1} and Figure \ref{fig:UW2-2d-ex-alpha-1000} show the transportation with $\alpha=1$ and $\alpha=1000$, respectively. The same phenomena can be observed as in 1D case from \textit{Experiment 1}. In other words, the geodesic with the spatially dependent source function with small $\alpha$ in Figure \ref{fig:UW2-2d-ex-alpha-1} behaves closer to the normalized (classical) optimal transport geodesic and the geodesic with large $\alpha$ in Figure \ref{fig:UW2-2d-ex-alpha-1000} behaves closer to the Euclidean geodesic.



\begin{figure}[H]
\begin{minipage}{0.32\linewidth}
\includegraphics[width=1\linewidth]{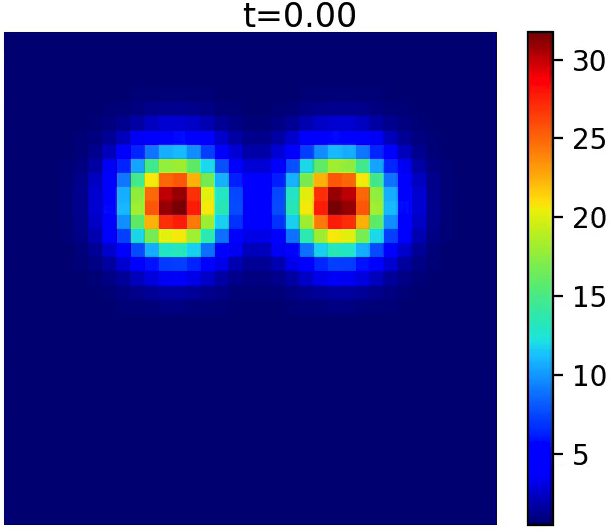}
\end{minipage}\hfill
\begin{minipage}{0.32\linewidth}
\includegraphics[width=1\linewidth]{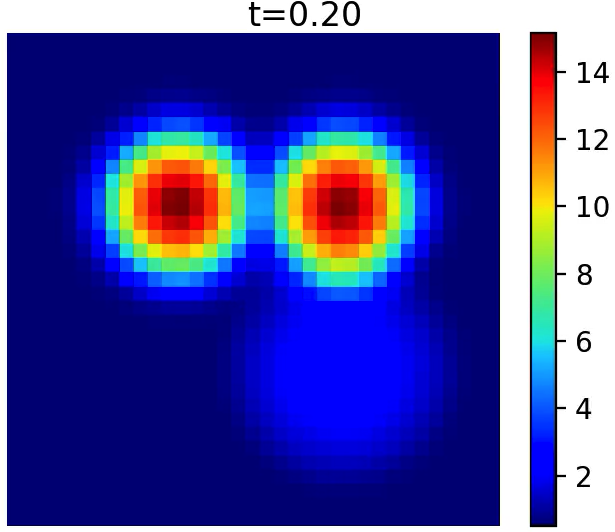}
\end{minipage}\hfill
\begin{minipage}{0.32\linewidth}
\includegraphics[width=1\linewidth]{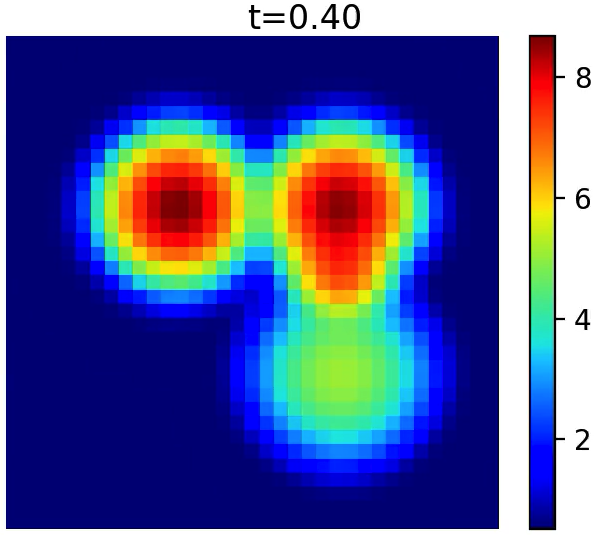}
\end{minipage}\hfill
\begin{minipage}{0.32\linewidth}
\includegraphics[width=1\linewidth]{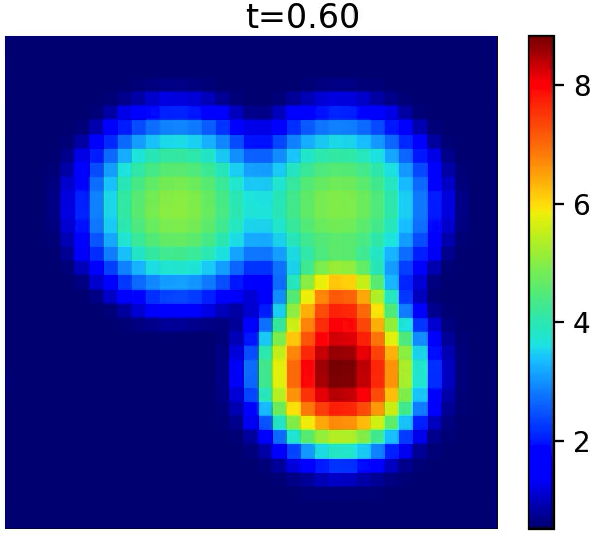}
\end{minipage}\hfill
\begin{minipage}{0.32\linewidth}
\includegraphics[width=1\linewidth]{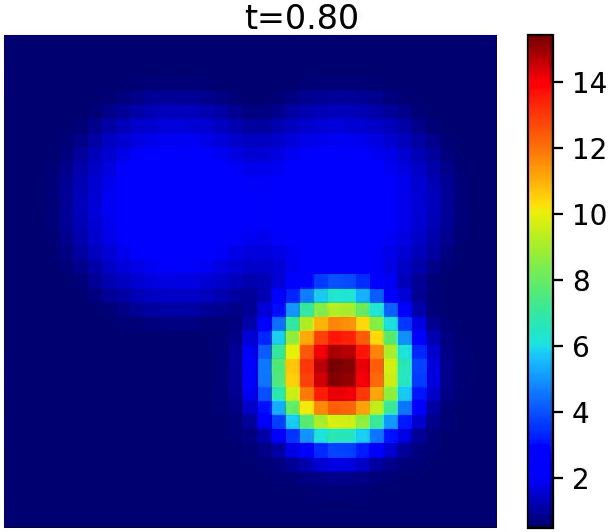}
\end{minipage}\hfill
\begin{minipage}{0.32\linewidth}
\includegraphics[width=1\linewidth]{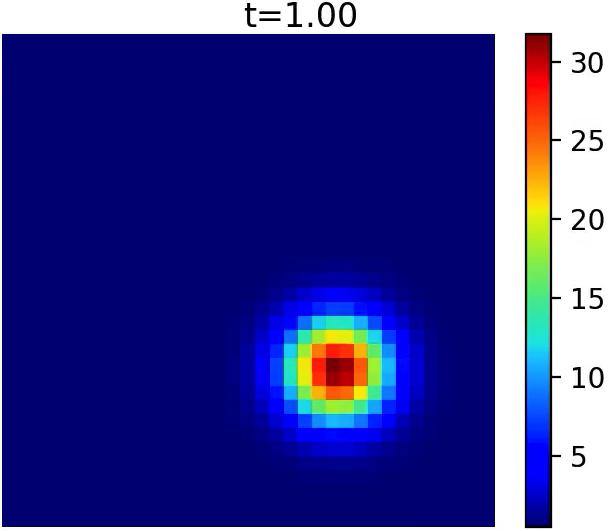}
\end{minipage}\hfill
\caption{\textit{Experiment 3.} $L^2$ generalized unnormalized optimal transportation: 2D example with a spatially dependent source function $f(t,x)$ and $\alpha=1000$.}
\label{fig:UW2-2d-ex-alpha-1000}
\end{figure}

\subsubsection{Experiment 4}
In this experiment, we are interested in calculating $L^2$ unnormalized Wasserstein distance between two images. Consider images of two cats with different total masses defined on the domain $\Omega=[0,1]\times [0,1]$.  We use Algorithm \ref{alg:gradient-descent} with the following parameters:
\begin{align*}
    N_x = 64, N_y=64, N_t = 15, \tau = 0.05.
\end{align*}
Figure \ref{fig:UW2-two-cats-05} and Figure \ref{fig:UW2-two-cats-1000} show transportation between two cats images with $\alpha=0.5$ and $\alpha=1000$, respectively.

\begin{figure}[ht]
\begin{minipage}{0.32\linewidth}
\includegraphics[width=1\linewidth]{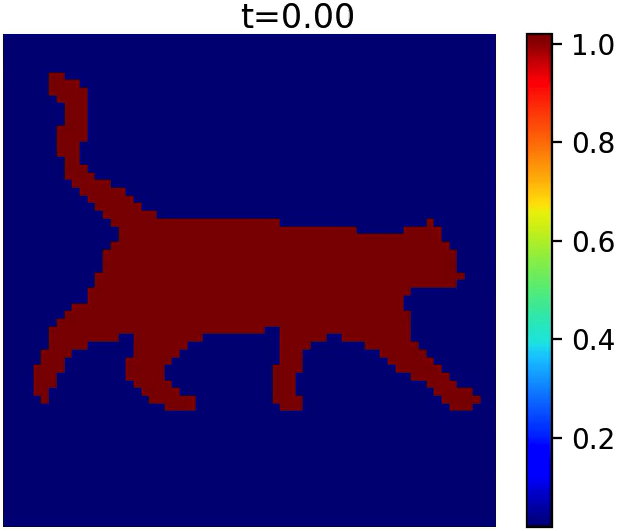}
\end{minipage}\hfill
\begin{minipage}{0.32\linewidth}
\includegraphics[width=1\linewidth]{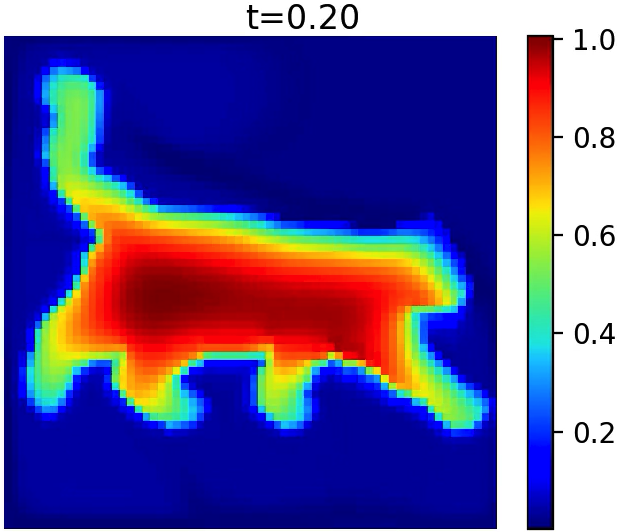}
\end{minipage}\hfill
\begin{minipage}{0.32\linewidth}
\includegraphics[width=1\linewidth]{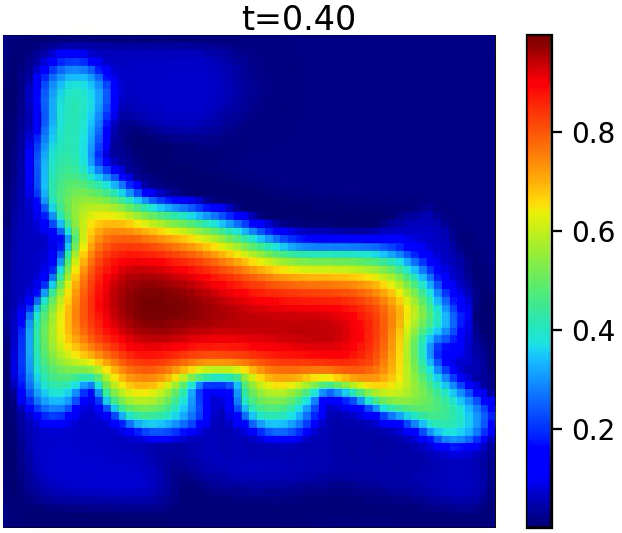}
\end{minipage}\hfill
\begin{minipage}{0.32\linewidth}
\includegraphics[width=1\linewidth]{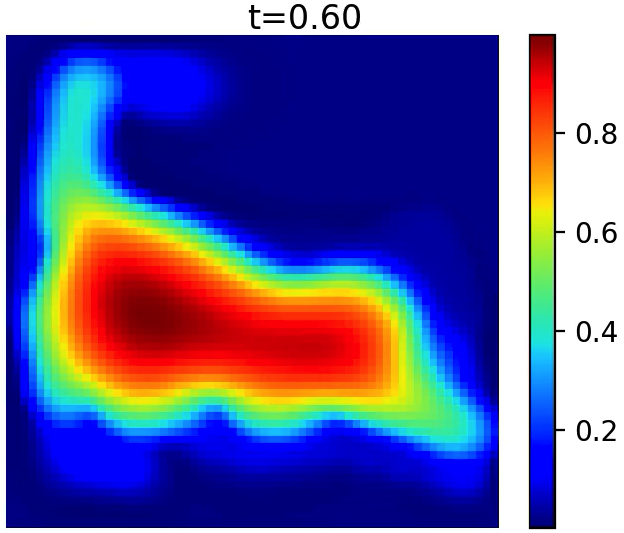}
\end{minipage}\hfill
\begin{minipage}{0.32\linewidth}
\includegraphics[width=1\linewidth]{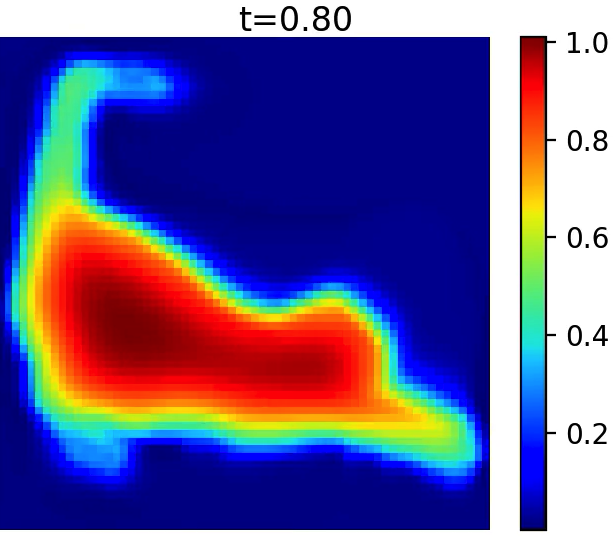}
\end{minipage}\hfill
\begin{minipage}{0.32\linewidth}
\includegraphics[width=1\linewidth]{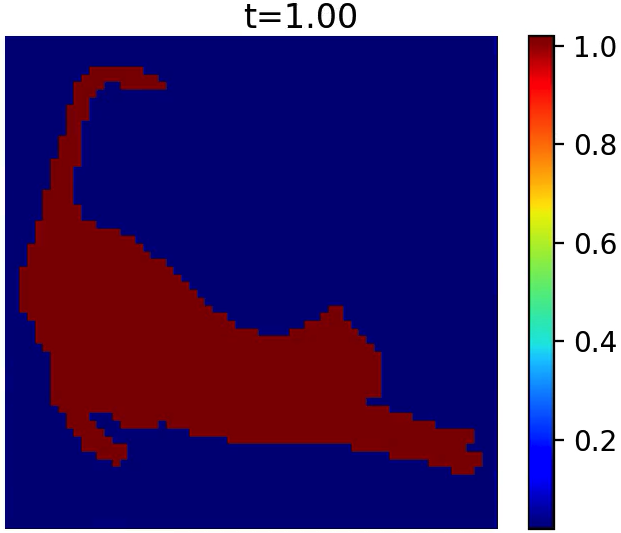}
\end{minipage}\hfill
\caption{\textit{Experiment 4.} $L^2$ generalized unnormalized optimal transportation: Two cats example with a spatially dependent source function $f(t,x)$ and $\alpha=0.5$.}
\label{fig:UW2-two-cats-05}
\end{figure}

\begin{figure}[ht]
\begin{minipage}{0.32\linewidth}
\includegraphics[width=1\linewidth]{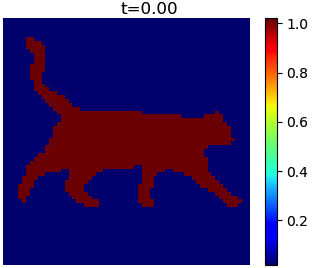}
\end{minipage}\hfill
\begin{minipage}{0.32\linewidth}
\includegraphics[width=1\linewidth]{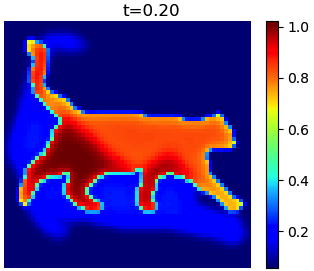}
\end{minipage}\hfill
\begin{minipage}{0.32\linewidth}
\includegraphics[width=1\linewidth]{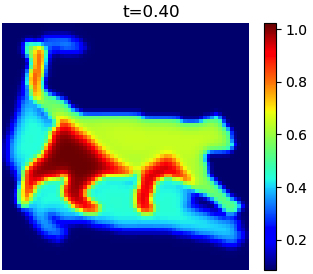}
\end{minipage}\hfill
\begin{minipage}{0.32\linewidth}
\includegraphics[width=1\linewidth]{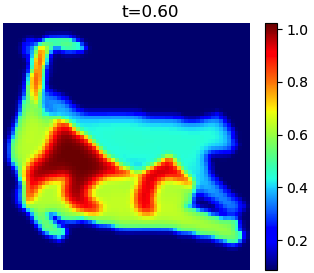}
\end{minipage}\hfill
\begin{minipage}{0.32\linewidth}
\includegraphics[width=1\linewidth]{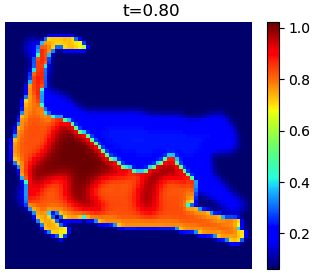}
\end{minipage}\hfill
\begin{minipage}{0.32\linewidth}
\includegraphics[width=1\linewidth]{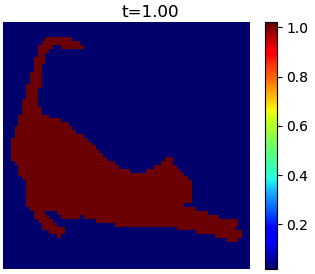}
\end{minipage}\hfill
\caption{\textit{Experiment 4.} $L^2$ generalized unnormalized optimal transportation: two cats example with a spatially dependent source function $f(t,x)$ and $\alpha=1000$.}
\label{fig:UW2-two-cats-1000}
\end{figure}

\subsection{Primal dual algorithm for $UW_1$} We conduct two numerical examples of $L^1$ unnormalized optimal transport using  Algorithm \ref{alg:L1-pdhg}.

\subsubsection{Experiment 5}
Assume $\Omega = [0,1] \times [0,1]$. Consider the two dimensional problem with the following initial densities:
\begin{align*}
    &\mu_0 = N\left((x,y), (\frac{1}{3},\frac{1}{2}), (\frac{1}{10},\frac{1}{10})\right)\\
    &\mu_1 = N\left((x,y), (\frac{2}{3},\frac{1}{2}), (\frac{1}{10},\frac{1}{10})\right) \cdot 1.4
\end{align*}
\noindent $N$ is the same as the one used in Experiment 3. We chose the parameters as:
$ N_x = N_y = 40, \epsilon = 0.001, \lambda = 0.0001, \tau = 0.01$. 
In Figure \ref{fig:UW1-ex1}, the initial densities $\mu_0$ and $\mu_1$ are shown on the top two plots and the minimizers $\bm{m}$'s are plotted for three different $\alpha$ values. Figure \ref{fig:UW1-ft} (a) shows the result from $L^1$ transportation with a spatially independent source function $f(t)$. 
This experiment shows the clear difference between $L^1$ unnormalized optimal transport with $f(t,x)$ and with $f(t)$. 
While the minimizer $\bm{m}$ from the unnormalized optimal transport with $f(t,x)$ is nonzero only on the area between two densities, the minimizer from the unnormalized optimal transport with $f(t)$ is nonzero everywhere. This is because the spatially dependent source function $f(t,x)$  affects the minimizer locally but the spatially independent source function $f(t)$ affects the minimizer globally.

\begin{figure}[h]
\begin{minipage}{0.48\linewidth}
\includegraphics[width=1\linewidth]{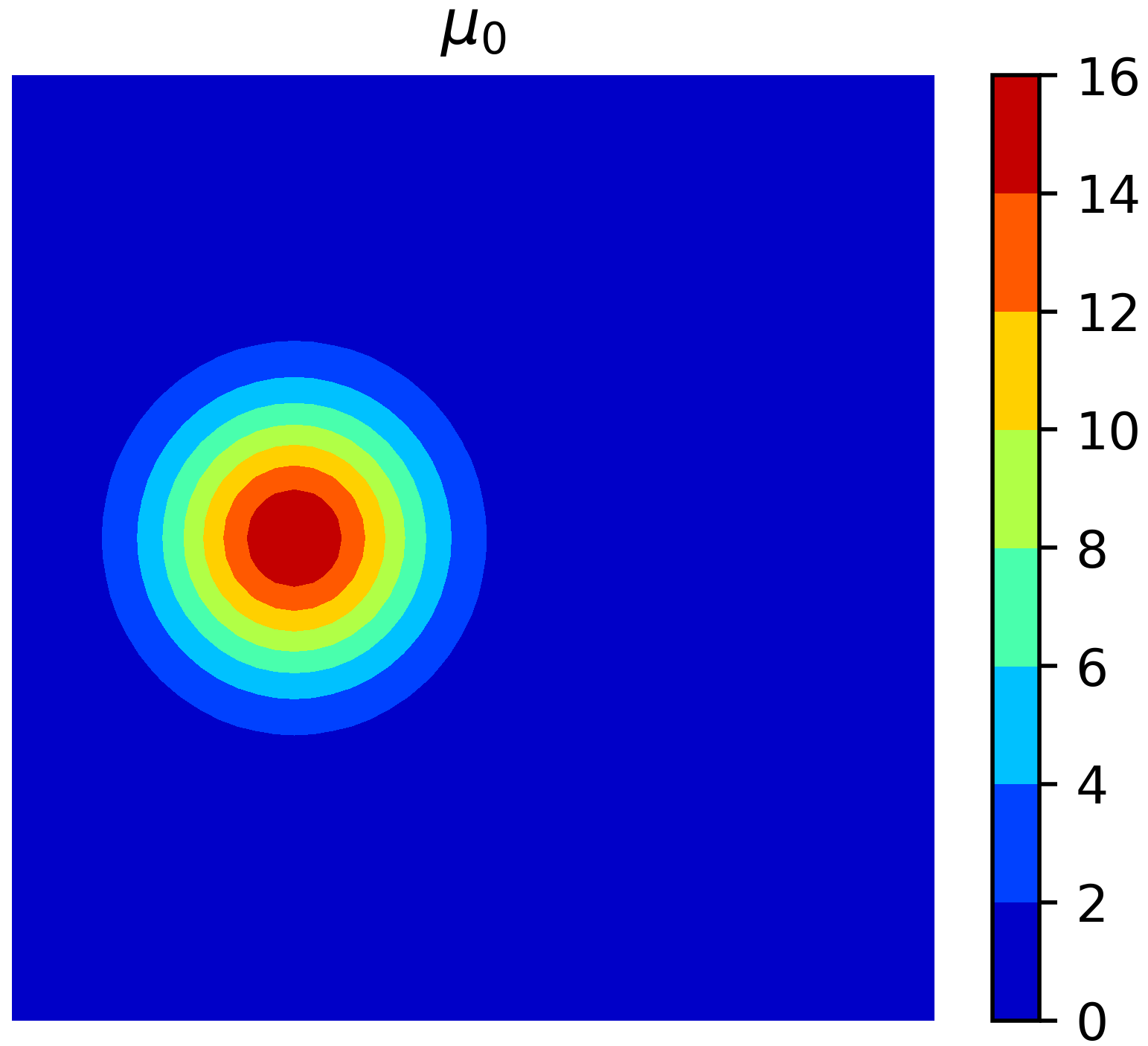}
\end{minipage}\hfill
\begin{minipage}{0.48\linewidth}
\includegraphics[width=1\linewidth]{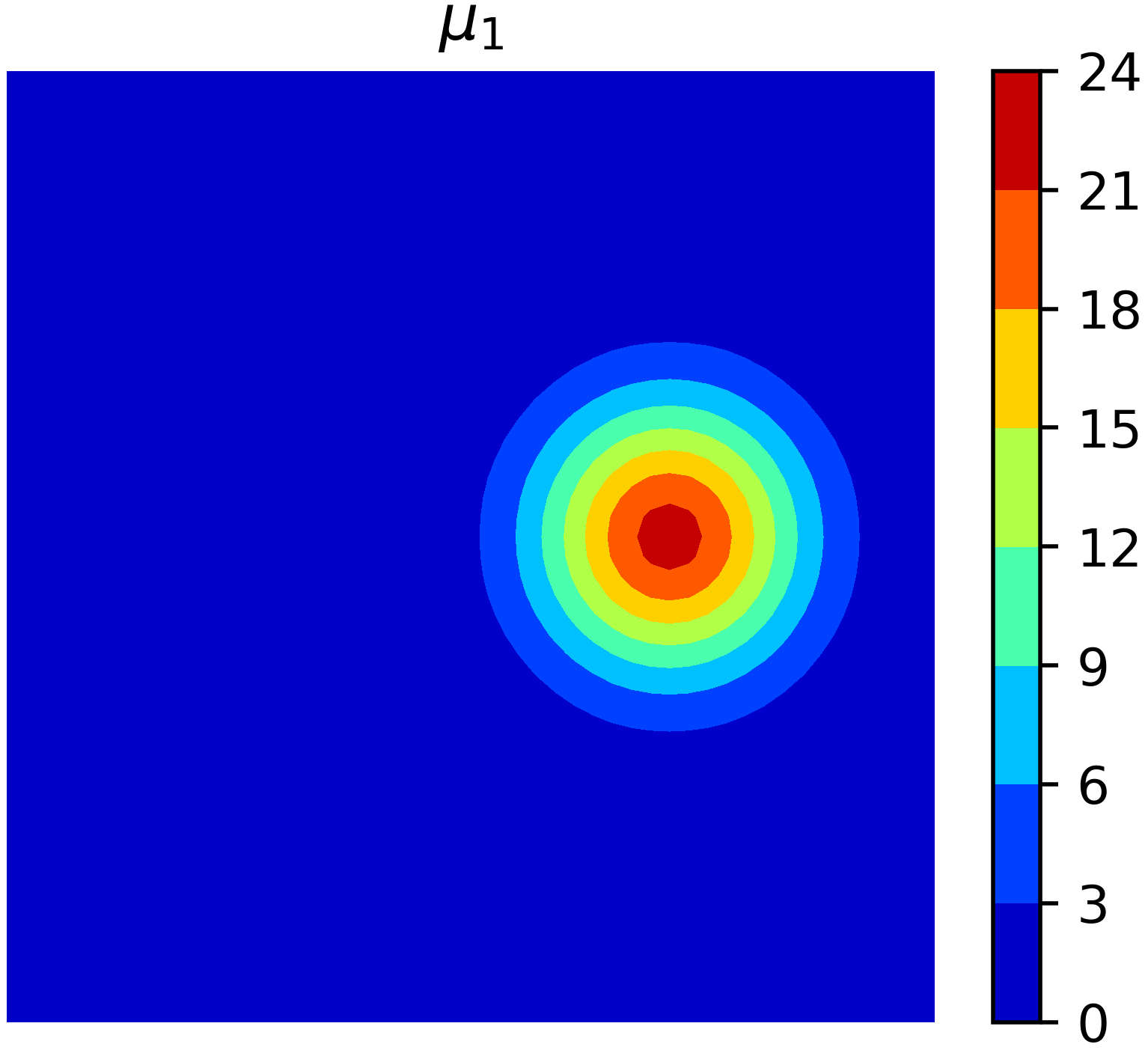}
\end{minipage}\hfill
\begin{minipage}{0.31\linewidth}
\includegraphics[width=1\linewidth]{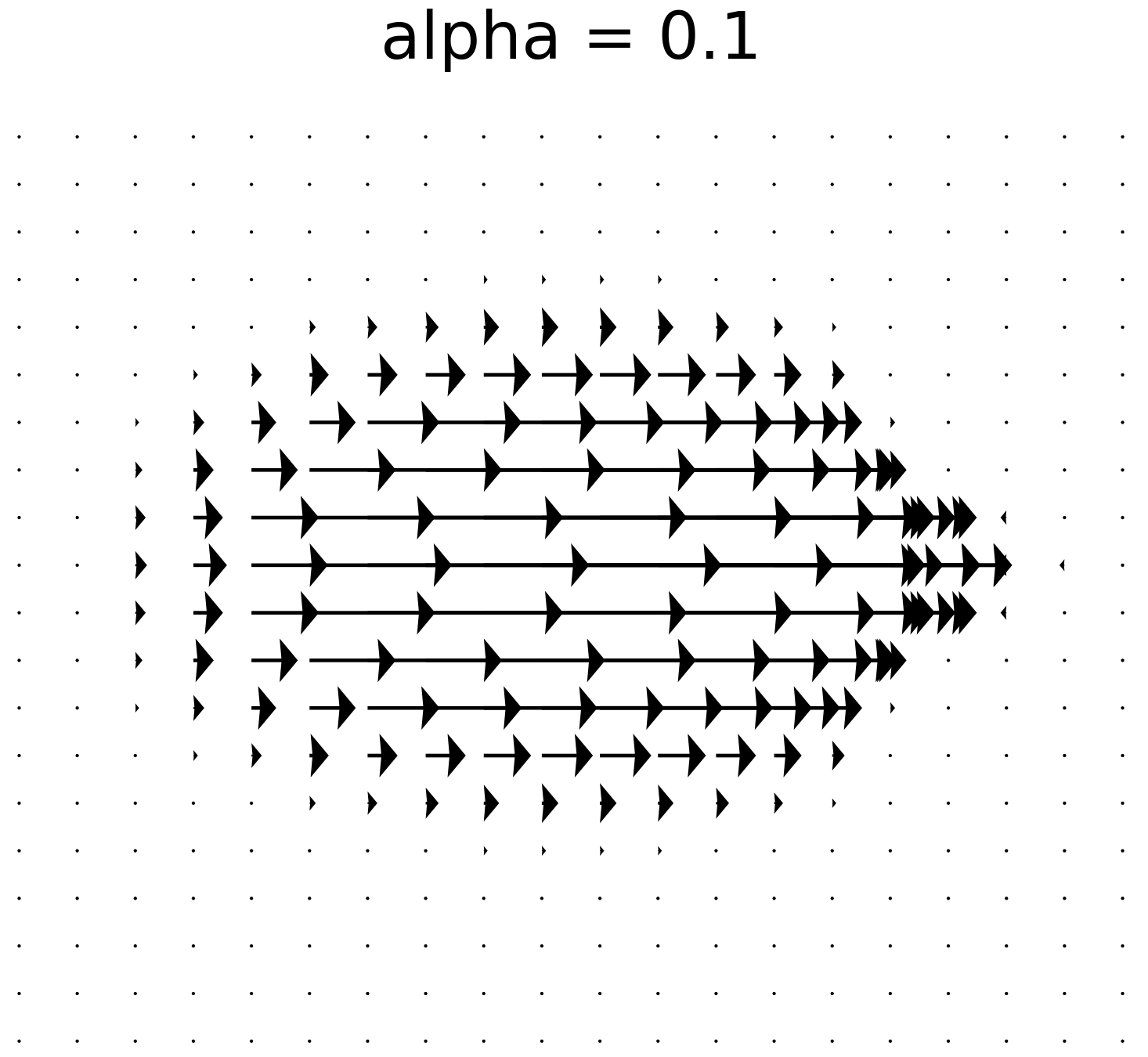}
\end{minipage}\hfill
\begin{minipage}{0.31\linewidth}
\includegraphics[width=1\linewidth]{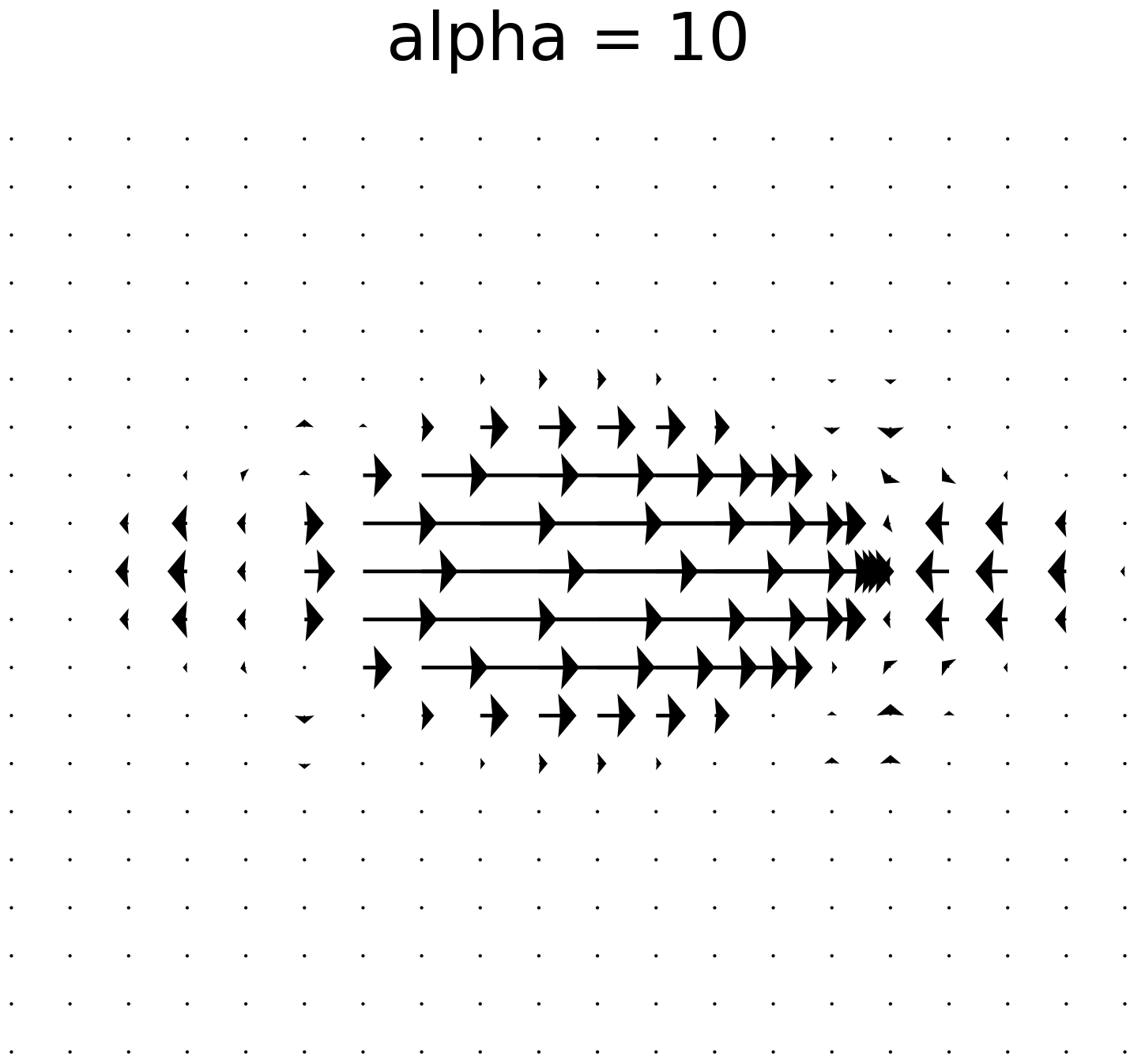}
\end{minipage}\hfill
\begin{minipage}{0.31\linewidth}
\includegraphics[width=1\linewidth]{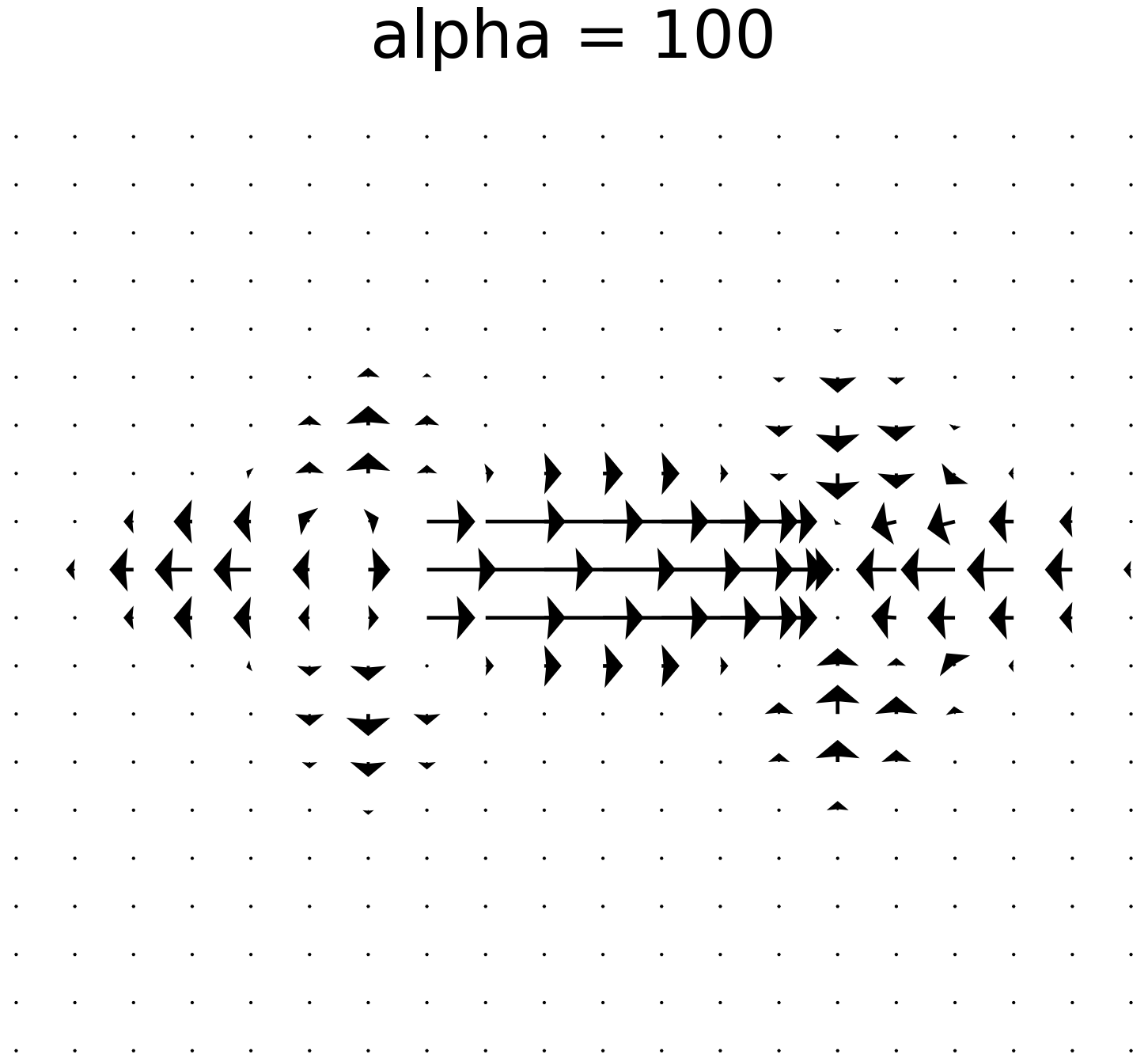}
\end{minipage}\hfill
\caption{\textit{Experiment 5.} $L^1$ unnormalized optimal transportation with $f(t,x)$ using different $\alpha$ values.}
\label{fig:UW1-ex1}
\end{figure}

\subsubsection{Experiment 6}
In this experiment, we are interested in $UW_2$ distance between two images. Consider the same 2D example as in the \textit{Experiment 4}. We use the Algorithm \ref{alg:L1-pdhg} with the following parameters:
\begin{align*}
    N_x = N_y = 256, \epsilon = 0.001, \lambda = 0.0001, \tau = 0.01.
\end{align*}
See Figure \ref{fig:UW1-ex2} to see the results of $L^1$ unnormalized optimal transport with a spatially dependent source function $f(t,x)$ with different $\alpha$ values $0.1$, $5$, and $10$. Figure \ref{fig:UW1-ft} (b) shows the result from $L^1$ transportation with a spatially independent source function $f(t)$. The result is similar to the \textit{Experiment 5}. 
The minimizer $\bm{m}$ from $L^1$ unnormalized optimal transport with $f(t)$ has nonzero values on the surrounding area of the two densities, but the minimizers from unnormalized optimal transport with $f(t,x)$ are zero on that surrounding area.


\begin{figure}[ht]
\begin{minipage}{0.48\linewidth}
\includegraphics[width=1\linewidth]{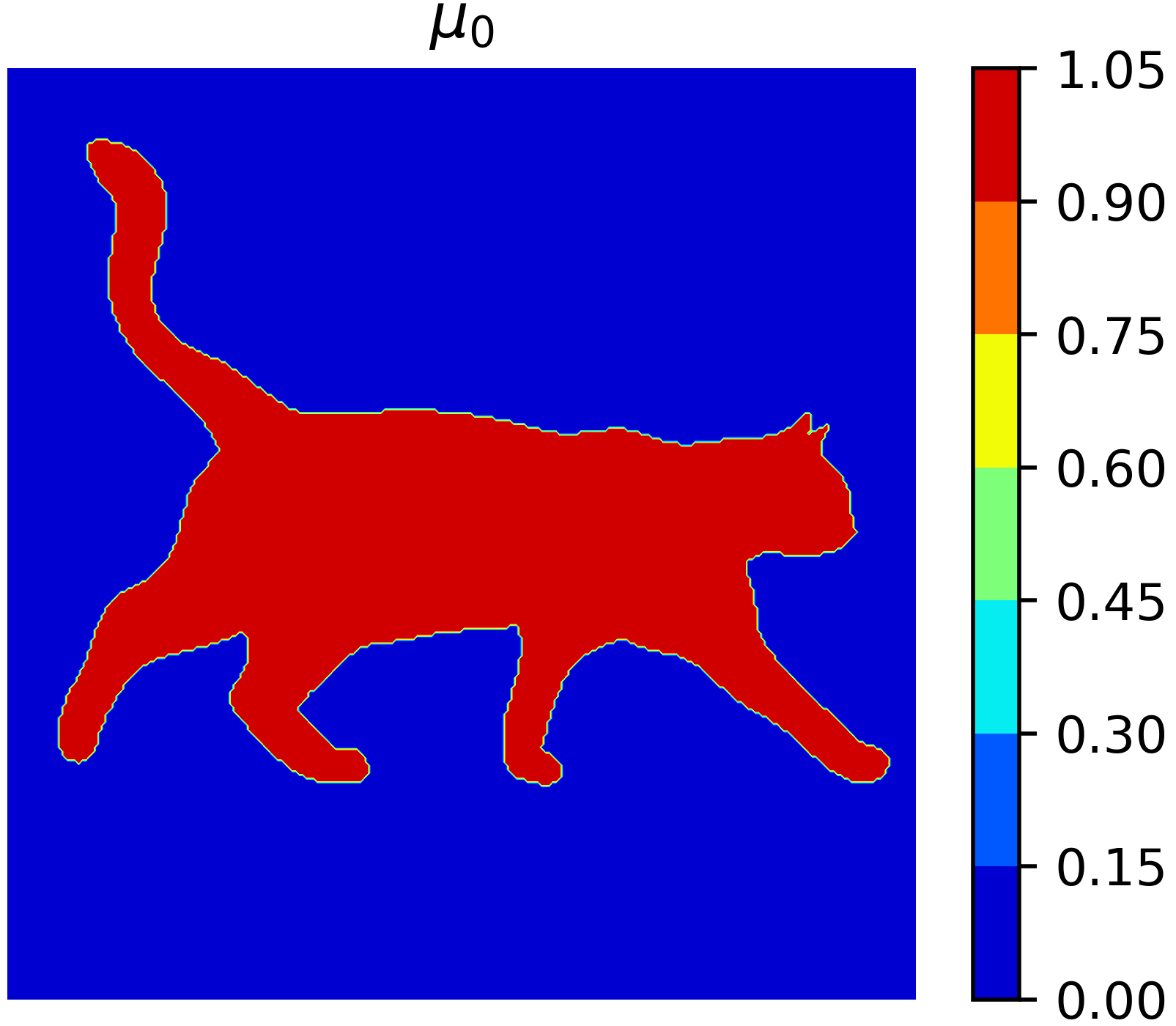}
\end{minipage}\hfill
\begin{minipage}{0.48\linewidth}
\includegraphics[width=1\linewidth]{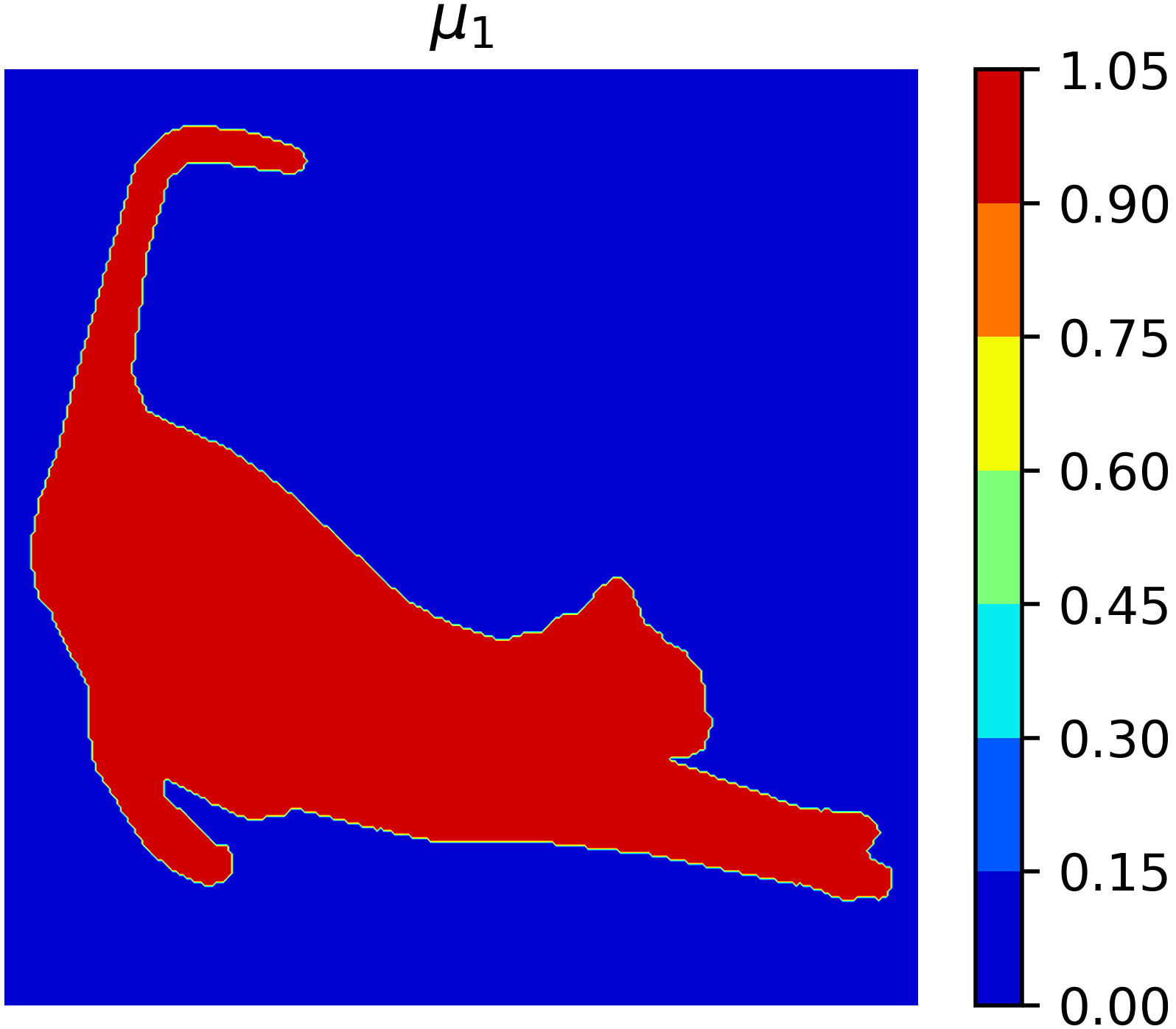}
\end{minipage}\hfill
\begin{minipage}{0.31\linewidth}
\includegraphics[width=1\linewidth]{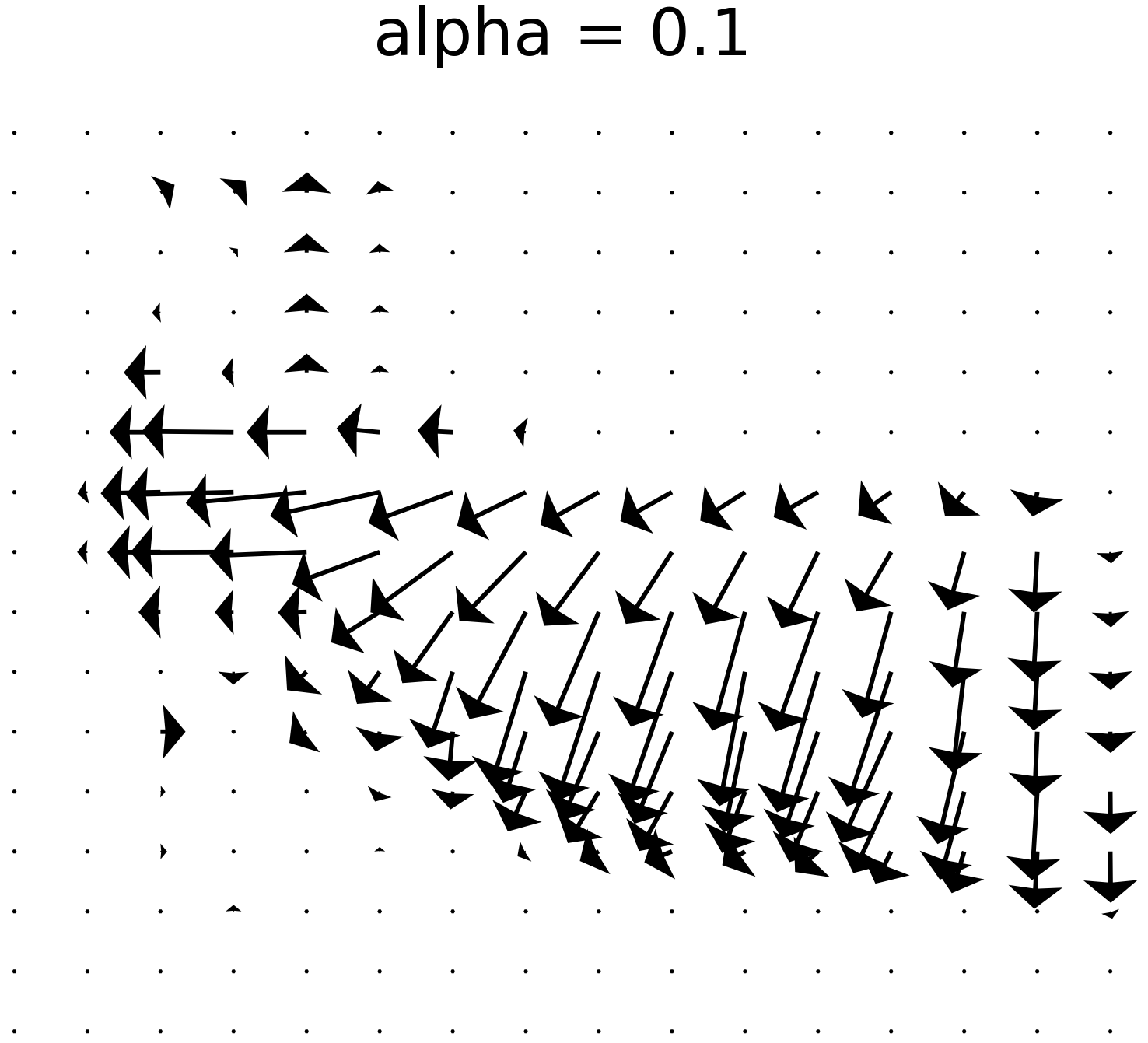}
\end{minipage}\hfill
\begin{minipage}{0.31\linewidth}
\includegraphics[width=1\linewidth]{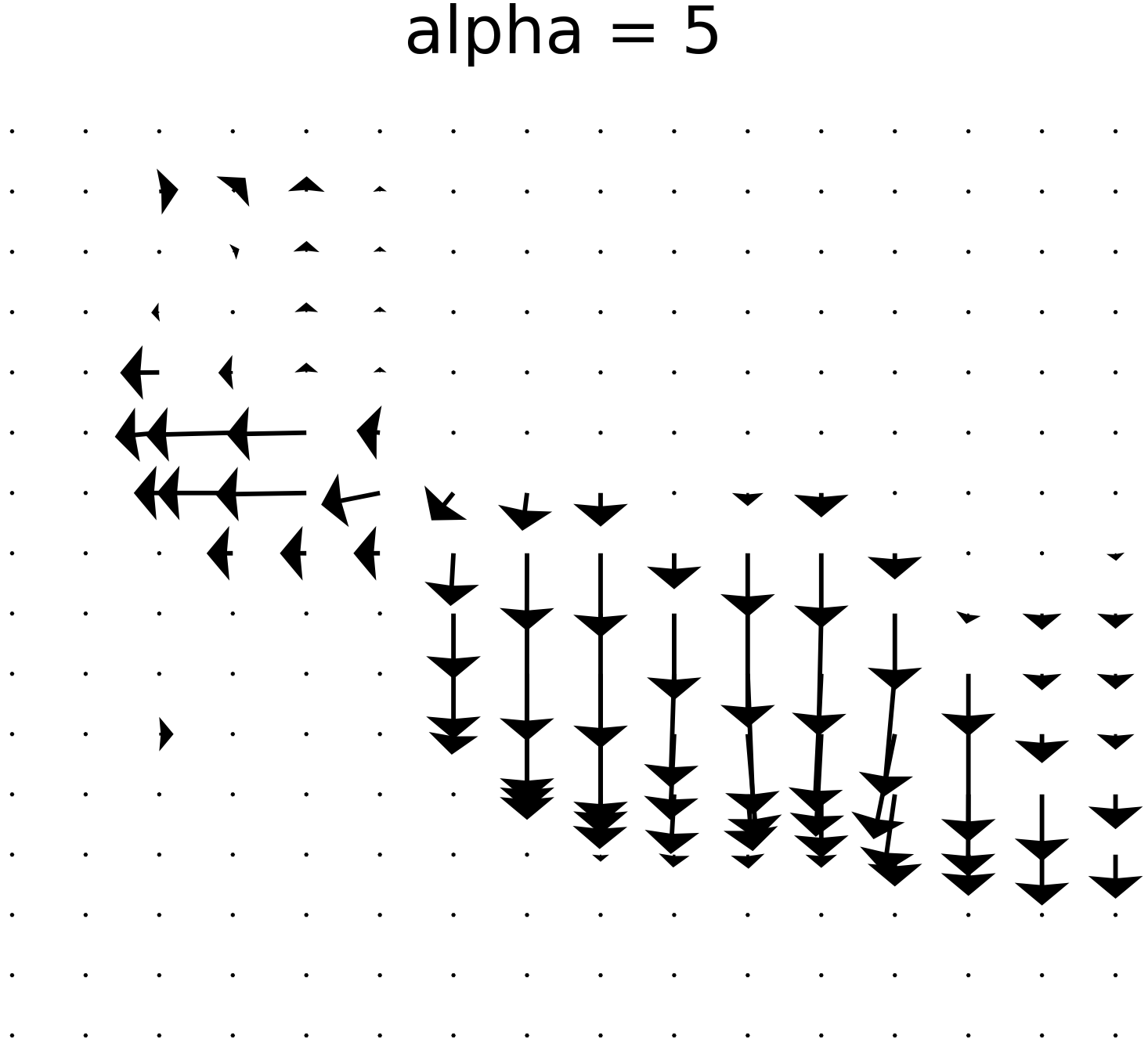}
\end{minipage}\hfill
\begin{minipage}{0.31\linewidth}
\includegraphics[width=1\linewidth]{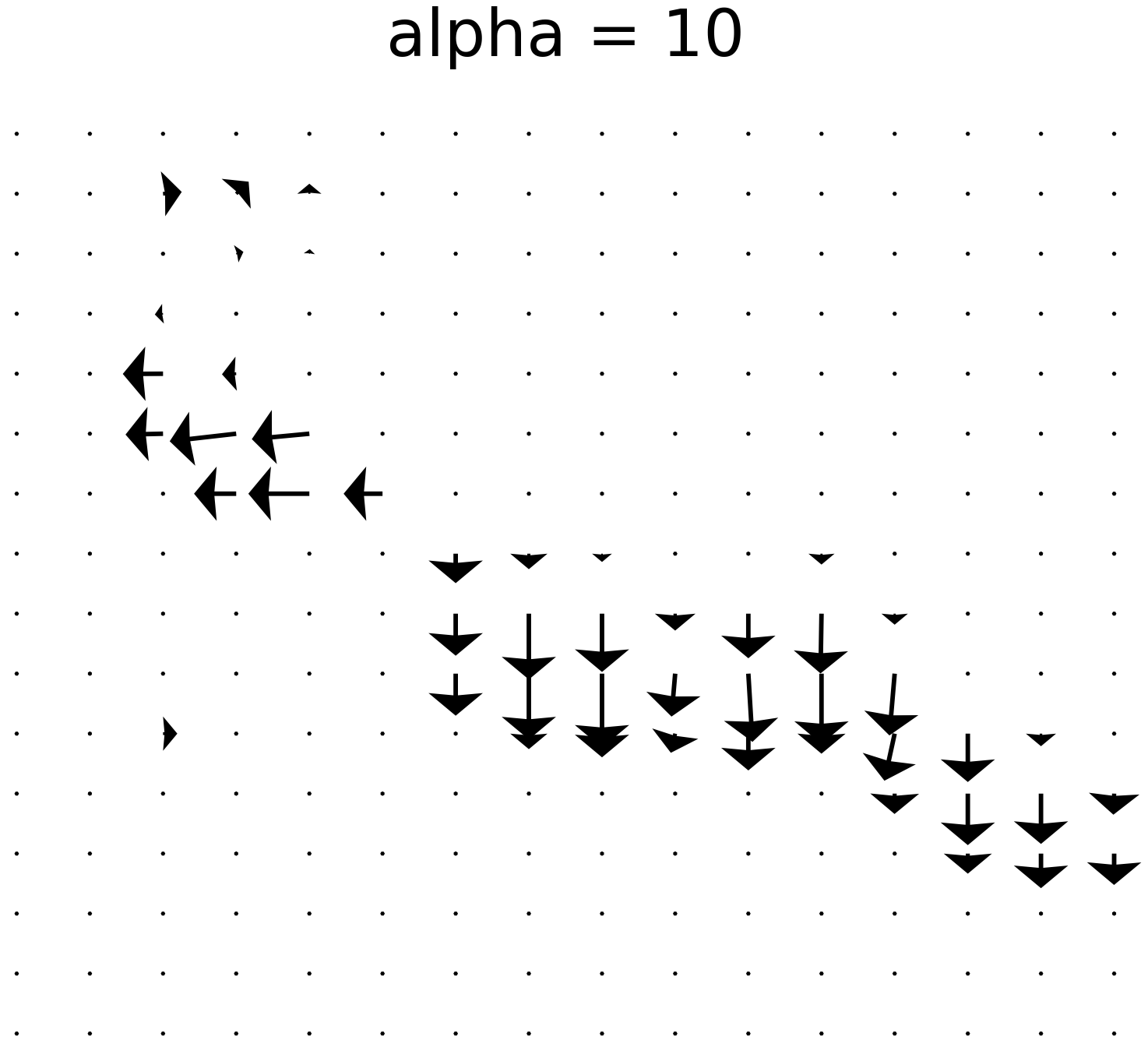}
\end{minipage}\hfill
\caption{\textit{Experiment 6.} $L^1$ unnormalized optimal transportation with $f(t,x)$ using different $\alpha$ values.}
\label{fig:UW1-ex2}
\end{figure}

\begin{figure}[ht]
\begin{minipage}{0.48\linewidth}
\includegraphics[width=1\linewidth]{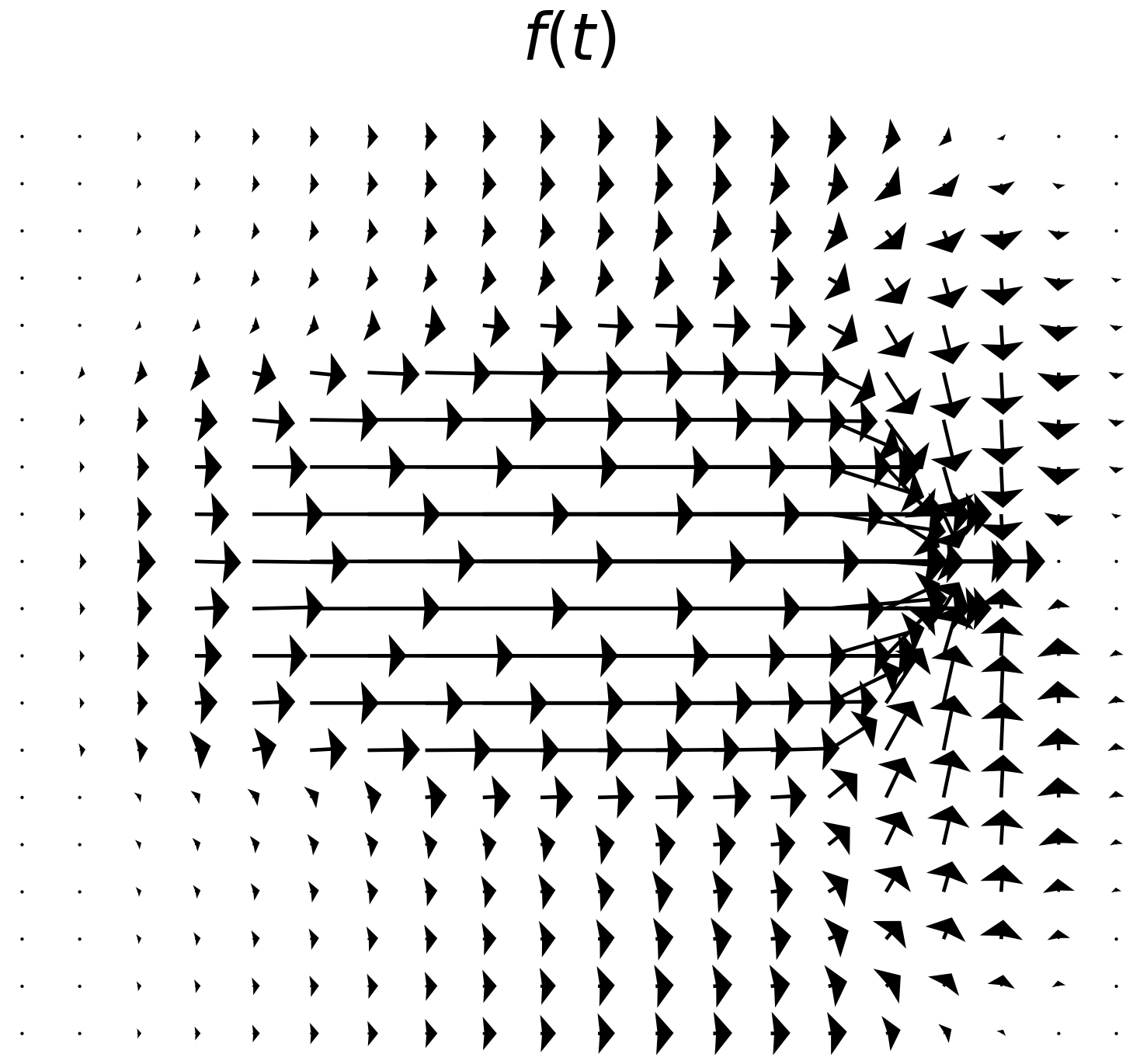}
\caption*{(a) $UW_1$ with $f(t)$ from \textit{Experiment 5}.}
\end{minipage}\hfill
\begin{minipage}{0.48\linewidth}
\includegraphics[width=1\linewidth]{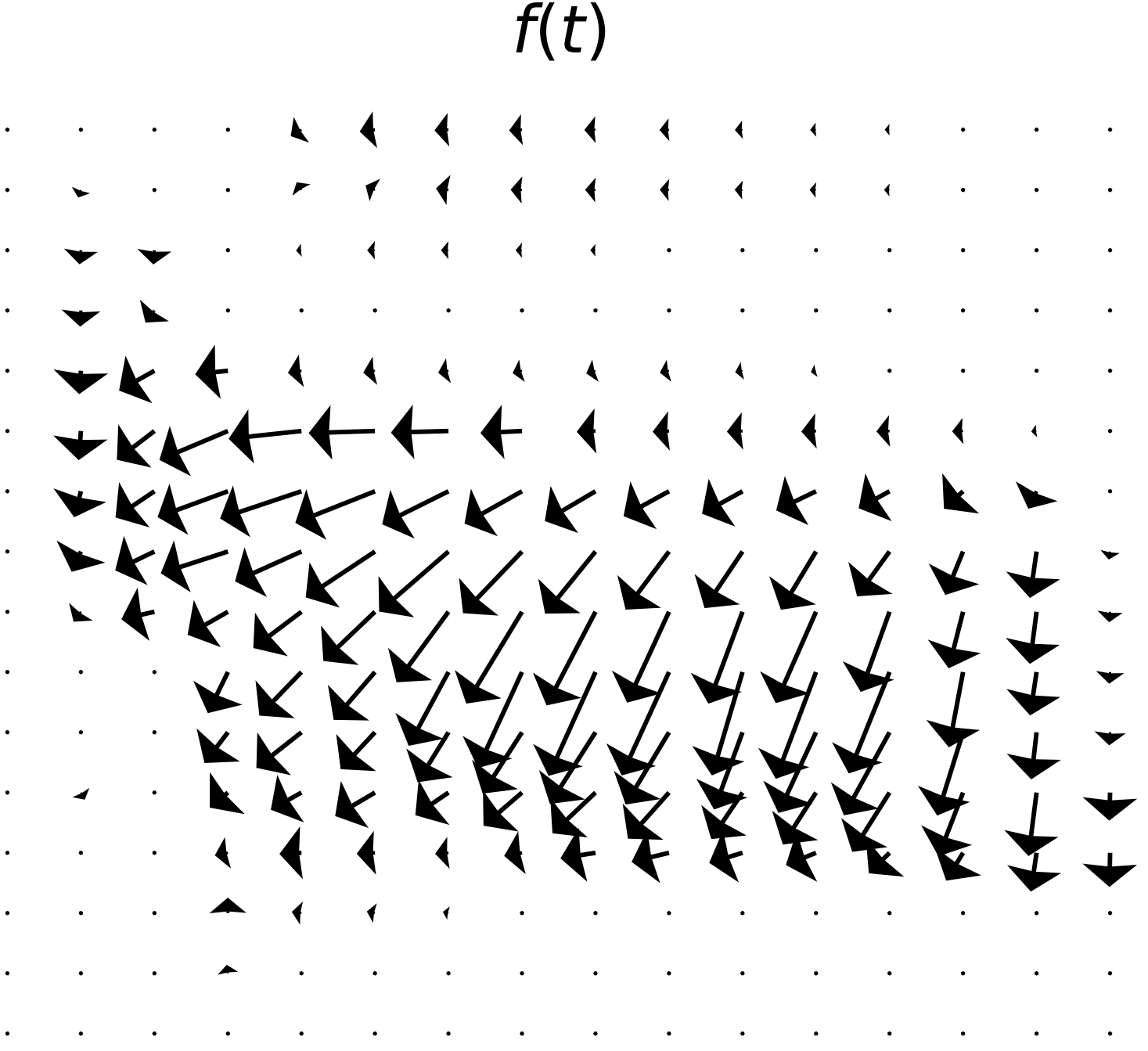}
\caption*{(b) $UW_1$ with $f(t)$ from \textit{Experiment 6}.}
\end{minipage}\hfill
\caption{$L^1$ unnormalized optimal transportation with a spatially independent source function $f(t)$.}
\label{fig:UW1-ft}
\end{figure}

\section{Discussion}
\label{sec:con}
In this paper, we introduced a new class of $L^p$ generalized unnormalized optimal transport distance with a spatially dependent source function. We presented new fast algorithms for $L^1$ and $L^2$ generalized unnormalized optimal transport. For $L^1$ case, we derived the Kantorovich duality and used a primal-dual algorithm which has explicit formulas with low computational costs. For $L^2$ case, we derived the duality formula, the generalized unnormalized Monge problem and corresponding Monge-Amp\`ere equation. We applied a weighted Laplacian operator $L_\mu$ to formulate the problem into an unconstrained optimization. The gradient operator of this unconstrained optimization is precisely the Hamilton-Jacobi equation. We apply the Nesterov accelerated gradient descent method to solve this minimization problem.

Our algorithm can be applied to general unnormalized/unbalanced optimal transport problems. It is also suitable for considering general variational mean-field games. In future works, we will derive new formulations for all related $L^p$ unbalanced or unnormalized mean-field games and design fast numerical algorithms to solve them.
\newpage

\end{document}